%% file: 0main.tex
\documentclass[12pt,a4paper,oneside]{amsart}
\include{top}

\usepackage[style=alphabetic, citestyle=alphabetic, sorting=nyvt]{biblatex}
\usepackage{csquotes}

\usepackage{tkz-euclide} 
\usetikzlibrary{matrix}

\makeatletter
\let\@wraptoccontribs\wraptoccontribs
\makeatother

\title[Relative filling functions]{Quasi-isometry invariance of relative filling functions}

\author{Sam Hughes}
\address{Sam Hughes, University of Oxford}
\email{sam.hughes@maths.ox.ac.uk}
\author{Eduardo Mart\'inez-Pedroza}
\address{Eduardo Mart\'inez-Pedroza, Memorial University of Newfoundland}
\email{emartinezped@mun.ca}
\author{Luis Jorge S\'anchez Salda\~na}
\address{Luis Jorge S\'anchez Salda\~na, Universidad Nacional Aut\'onoma de M\'exico}
\email{luisjorge@ciencias.unam.mx}
\contrib[with an appendix by]{Ashot Minasyan}
\address{Ashot Minasyan, University of Southampton}
\email{aminasyan.gmail.com}
\date{\today}


\begin{document}

\maketitle
\begin{abstract}
    For a finitely generated group $G$ and collection of subgroups $\mathcal{P}$ we prove that the relative Dehn function of a pair $(G,\mathcal{P})$ is invariant under quasi-isometry of pairs.  Along the way we show quasi-isometries of pairs preserve almost malnormality of the collection and fineness of the associated coned off Cayley graphs.  We also prove that for a cocompact simply connected combinatorial $G$-$2$-complex $X$ with finite edge stabilisers, the combinatorial Dehn function is well-defined if and only if the $1$-skeleton of $X$ is fine.
    
    We also show that if $H$ is a hyperbolically embedded subgroup of a finitely presented group $G$, then the relative Dehn function of the pair  $(G, H)$ is well-defined. In the appendix, it is shown that  show that the Baumslag-Solitar group $\mathrm{BS}(k,l)$ has a well-defined Dehn function with respect to the cyclic subgroup generated by the stable letter if and only if neither $k$ divides $l$ nor $l$ divides $k$.
\end{abstract}
 
\section{Introduction}
The main objects of study in this article are pairs $(G,\calp)$ where $G$ is a \textbf{finitely generated group} with a chosen word metric $\dist_G$, and $\calp$ is a \textbf{finite collection of subgroups}, \emph{note that these assumptions will stand throughout the introduction.}

Let $\Hdist_G$ denote the Hausdorff distance between subsets of $G$, and let $G/\calp$ denote the collection of left cosets $gP$ for $g\in G$ and $P\in\calp$.

For constants $L\geq 1$, $C\geq 0$ and $M\geq 0$, an \emph{$(L,C,M)$-quasi-isometry of  pairs}  $q\colon (G, \mathcal{P})\to (H, \mathcal{Q})$ is an $(L,C)$-quasi-isometry $q\colon G \to H$ such that  the relation
\begin{equation*}
\{ (A, B)\in G/\mathcal{P} \times H/\mathcal{Q} \colon \Hdist_H(q(A), B) <M  \} \end{equation*}
satisfies that the projections to $G/\mathcal{P}$ and $H/\mathcal{Q}$ are surjective. 

This article is part of the program of investigating which properties of pairs $(G,\calp)$ are invariant under quasi-isometry of pairs. There are recent results in this direction.
For example, it is a consequence of the quasi-isometric rigidity of relative hyperbolicity~\cite{BDM09}, that if $(G,\calp)$ is a relatively hyperbolic pair, $\calp$ is a collection of non-relatively hyperbolic groups, and  $(G,\calp)$ and $(H,\calq)$ are quasi-isometric pairs, then $H$ is hyperbolic relative to $\calq$. Under natural assumptions, quasi-isometries of pairs between relatively hyperbolic pairs induce  canonical homeomorphisms between their Bowditch boundaries~\cite{BuHr21} and canonical isomorphisms of JSJ trees~\cite{HaHr19}. Outside the framework of relatively hyperbolic groups,   
it is known that quasi-isometries of pairs preserve the number of Bowditch's filtered ends~\cite{MaSa21}.  For a recent survey we direct the reader to \cite{HMS2021}.

For a pair $(G, \calp)$, Osin introduced the notions of \emph{finite relative presentation} and \emph{relative Dehn function} $\Delta_{G,\calp}$ as natural  generalizations of their standard counterparts for finitely generated groups, see~\cite{Osin06}. These notions characterise relatively hyperbolic pairs $(G,\calp)$ as the ones which are relatively finitely presented and have relative Dehn function bounded from above by a linear function. By quasi-isometric rigidity of relative hyperbolicity, among relatively finitely presented pairs,   quasi-isometries of pairs preserve  having linear relative Dehn function.

The main result of this article confirms the natural expectation that among relatively finitely presented pairs, quasi-isometric pairs have equivalent relative Dehn functions. This is not an elementary statement, as we describe below.

\begin{convention}[$\Delta_{G,\calp}$  is well-defined]
By \emph{$\Delta_{G,\calp}$  is well-defined} we mean that $G$ is finitely presented relative to $\calp$ and the relative Dehn function $\Delta_{G,\calp}$ takes only finite values with respect to a finite relative presentation of $G$ and $\calp$. From here on, when we refer to a relative Dehn function, we always assume that it has been defined using a finite relative presentation.
\end{convention}

Let $\calp$ be a collection of subgroups of group $G$.  A \emph{refinement} $\calp^\ast$ of $\calp$ is a set of representatives of conjugacy classes of the collection of subgroups $\{\Comm_G(gPg^{-1}) \colon P\in\calp \text{ and } g\in G \}$ where $\Comm_G(P)$ denotes the commensurator of the subgroup $P$ in $G$.

\begin{thmx}\label{thmx:DehnQI}
Let $(G,\calp)\to(H, \calq)$ be a quasi-isometry of pairs and let $\calp^*$ be a refinement of $\calp$. If the relative Dehn function $\Delta_{H,\calq}$ is well-defined, then $\Delta_{G,\calp^*}$ is well-defined and $\Delta_{G,\calp^\ast} \asymp \Delta_{H,\calq}$.
\end{thmx} 

A phenomenon that occurs for pairs $(G,\calp)$ is that being relatively finitely presented does not imply that the relative Dehn function is well-defined. This is in sharp contrast with the standard framework where a finitely presented group always has a well-defined Dehn function. The proof of Theorem~\ref{thmx:DehnQI} provides an insight into this phenomenon via the following  results on which our argument relies on.

In the framework of relatively hyperbolic groups, Bowditch introduced the notion of fine graph~\cite{Bo12}. A \emph{circuit} in a simplicial graph is an embedded close path. A simplicial graph $\Gamma$ is \emph{fine} if for every $n\geq 0$ and every edge $e$ in $\Gamma$ there are finitely many circuits of length less than or equal to $n$ which contain $e$. This is weaker than the graph being locally finite. The relationship between this notion and isoperimetric functions was made explicit by Groves and Manning~\cite[Proposition~2.50, Question~2.51]{GM08}. The following result can be interpreted as a homotopical version of ~\cite[Theorem~1.3]{MP16} where an analogous statement is proved for homological Dehn functions.

\begin{thmx}[\Cref{thm.relDehn2complex}]\label{thmx.relDehn2complex}
Let $X$ be a cocompact simply connected combinatorial $G$-$2$-complex with finite edge stabilisers.   The combinatorial Dehn function $\Delta_X$ of $X$ takes only finite values if and only if the $1$-skeleton of $X$ is a fine graph.
\end{thmx} 
 
It is obvious that being fine is not a property preserved by quasi-isometries in the class of graphs. For a pair $(G,\calp)$,  together with a finite generating set $S$ of $G$, one can assign a connected and cocompact $G$-graph known as the coned-off Cayley graph $\hat\Gamma(G,\calp,S)$; a  notion introduced by Farb~\cite{farb}, see \Cref{def:conedOffCayleyGraph}. It is an observation that the quasi-isometry type of $\hat\Gamma(G,\calp,S)$ is independent of the finite generating set $S$; \emph{throughout the   introduction  $\hat\Gamma(G,\calp)$ denotes the coned-off Cayley graph with respect to some finite generating set of $G$}. In this framework, under some assumptions,  we are able to prove that fineness is preserved under quasi-isometry of pairs in the class of coned-off Cayley graphs. 
A collection of subgroups $\calp$ of a group $G$ is \emph{reduced} if  
for any $P,Q\in \calp$ and $g\in G$, then $P$ and $gQg^{-1}$ being commensurable subgroups  implies $P=Q$ and $g\in P$.

 \begin{thmx}[\Cref{thm:fine}]\label{thmx:fine}
Let $q\colon (G, \calp ) \to (H, \calq)$ be a quasi-isometry of pairs.
Suppose  $\calp$ and $\calq$ are reduced. Then there is an induced  quasi-isometry of graphs $\hat q\colon \hat \Gamma (G,\calp) \to \hat \Gamma (H,\calq)$, and if $\hat \Gamma (H,\calq)$ is a fine graph then $\hat \Gamma (G,\calp)$ is a fine graph.
\end{thmx}

The condition that the coned-off Cayley graph $\hat\Gamma(G,\calp)$ is   fine forces the collection $\calp$ to be almost malnormal (see Definition~\ref{def:malnormal}). It is an observation that any almost malnormal collection of infinite subgroups is reduced. We prove that the property of being almost malnormal is preserved under quasi-isometry of pairs up to taking a refinement.  

\begin{thmx}[\Cref{thm:malnormal}]\label{thmx:malnormal}
 Let $q\colon (G, \mathcal{P})\to (H, \mathcal{Q})$ be a quasi-isometry of pairs.  If $\calq$ is an almost malnormal collection of infinite  subgroups, then any refinement $\calp^*$ of $\calp$ is almost malnormal and $q\colon (G,\calp^*) \to (H,\calq)$ is a quasi-isometry of pairs.
\end{thmx}

The previous results can be linked to Osin's definition  of relative Dehn function $\Delta_{G,\calp}$ of a relatively finitely presented pair $(G,\calp)$ via the following result. A connected graph $\Gamma$ is called  \emph{fillable} if, when considering $\Gamma$ with the length metric obtained by regarding each edge as an segment  of length one, there is an integer $k$ such that the coarse isoperimetric function $f_k^\Gamma$ takes only finite values, see Section~\ref{sec:coarse:isop.functions} for definitions. 

\begin{thmx}[See \Cref{thm:ConefOff}]\label{thmx:ConefOff}
If $(G,\calp)$ is a relatively finitely presented pair, then
\begin{enumerate}
    \item \label{thmx:ConefOff:1} $\hat\Gamma(G,\calp)$ is fillable. 
    \item \label{thmx:ConefOff:2} The relative Dehn function $\Delta_{G,\calp}$ is well-defined if and only if  $\hat\Gamma(G,\calp)$ is fine graph. 
\end{enumerate}
Conversely, if $\hat\Gamma(G,\calp)$ is   fine and fillable, then $(G,\calp)$ is a relatively finitely presented pair and hence $\Delta_{G,\calp}$ is well-defined.
\end{thmx}

The following result is a re-statement of a result of Osin~\cite[Theorem 2.53]{Osin06}, see \Cref{prop:Osin}. This statement allow us translate his definition of relative Dehn function to the realm of coarse isoperimetric functions  of coned-off Cayley graphs.

\begin{thmx}[Osin]\label{thmx:Osin}
Let $G$ be a group and let $\calp$ be a collection of subgroups. Suppose that $\Delta_{G,\calp}$ is well-defined. Then $\Delta_{G,\calp}$ is equivalent  to the coarse isoperimetric function $f^{\hat\Gamma(G,\calp)}_N$ of $\hat\Gamma(G,\calp)$ for all sufficiently large integers $N$.
\end{thmx}

Let us describe the argument proving  \Cref{thmx:DehnQI} using the  results that have been stated.

\begin{proof}[Proof of Theorem~\ref{thmx:DehnQI}]
Let us first observe that we can
assume that the collections $\calp$ and $\calq$ contain only infinite subgroups. First note that  if $\calp_\infty$ and $\calq_\infty$ are the collections obtained by removing finite subgroups from $\calp$ and $\calq$ respectively, then $q\colon (G,\calp_\infty)\to(H, \calq_\infty)$ is a quasi-isometry of pairs as well.  Moreover, for an arbitrary pair $(K,\call)$,  adding or removing a finite subgroup of $K$ to $\call$ preserves having well-defined relative Dehn function, and if the functions are well-defined they are equivalent,  see for example~\cite[Theorem 2.40]{Osin06}.  
 
Assume that $\calp$ and $\calq$ consist only of infinite subgroups.
Since $(H,\calq)$ is relatively finitely presented and $\Delta_{H,\calq}$ is well-defined, \Cref{thmx:ConefOff} implies that $\hat\Gamma(H,\calq)$ is fillable and fine. 
Since $\hat\Gamma(H,\calq)$ is a fine graph, it follows that $\calq$ is an almost malnormal collection.  
Then \Cref{thmx:malnormal} implies that $\calp^*$ is an almost malnormal collection. Hence, both $\calq$ and $\calp^*$ are reduced collections and $q\colon (G,\calp^*) \to(H,\calq)$ is a quasi-isometry of pairs. Now we can invoke \Cref{thmx:fine} to obtain a quasi-isometry $\hat q\colon \hat\Gamma(G,\calp^*) \to \hat\Gamma(H,\calq)$ and also obtain that $\hat\Gamma(G,\calp^*)$ is fine. It is an standard result in the literature that being fillable is a property preserved by quasi-isometry in the class of connected graphs, and any two quasi-isometric graphs have equivalent coarse isoperimetric functions~(see for instance \cite[Proposition~III.H.2.2]{BridsonHaefligerBook}). The quasi-isometry $\hat q$ implies that  $\hat\Gamma(G,\calp^*)$ is fillable and  both  $\hat\Gamma(G,\calp^*)$ and  $\hat\Gamma(H,\calq)$ have equivalent coarse isoperimetric inequalities. Then \Cref{thmx:ConefOff} implies that $(G,\calp^*)$ is relatively finitely presented and $\Delta_{G,\calp^*}$ is well-defined. The proof concludes by invoking \Cref{thmx:Osin}.
\end{proof}

In the class of finitely generated groups, being finitely presented is a quasi-isometry invariant. We do not know the answer to the following general question:
\begin{question}
Suppose that $q\colon (G,\calp) \to (H,\calq)$ is a quasi-isometry of pairs and  $(H,\calq)$ is relatively finitely presented. Is $(G,\calp)$ relatively finitely presented?
\end{question}

There is a rich class of  pairs $(G,\calp)$ with well-defined relative Dehn function.
Hyperbolically embedded subgroups were introduced in~\cite{DGOsin2017} by Dahmani, Guirardel and Osin.  Given a group $G$, $X\subset G$ and $H\leq G$, let $H\hookrightarrow_h (G,X)$ denote that $H$ is a hyperbolically embedded subgroup of $G$ with respect to $X$. 


\begin{thmx}[\Cref{thm:last}]\label{thmx:last}
Let $G$ be a finitely presented group and $H\leq G$ be a subgroup. If $H\hookrightarrow_h G$ then the relative Dehn function $\Delta_{G,H}$ is well-defined. 
\end{thmx}

 In the context of \Cref{thmx:last}, the relative Dehn function $\Delta_{G,\calp}$ is bounded from above by a linear function if and only if $G$ is hyperbolic relative to $H$, see~\cite{Osin06}. It is well known that the class of pairs $(G,H)$ such that $H\hookrightarrow_h G$ properly extends relative hyperbolicity, for examples see~\cite{DGOsin2017}. 

In a preliminary version of this manuscript, we asked whether there exist pairs $(G,\calp)$ such that $\Delta_{G,\calp}$ is well-defined, but $\calp$ is not hyperbolically embedded in $G$. In this regard, consider the Baumslag-Solitar groups  $\mathrm{BS}(k,l)=\langle a, t\ |\ ta^kt^{-1}=a^l\rangle$, where $k,l \in \Z\setminus\{0\}$. In  \Cref{prop.BS.nwd} we show that $\mathrm{BS}(k,l)$ does not have a well-defined Dehn function with respect to the cyclic subgroup generated by the stable letter $t$ if either $k$ divides $l$ or $l$ divides $k$.  On the other hand, in the appendix Ashot Minasyan shows that the converse holds, that is, if neither $k \nmid l$ nor $l \nmid k$ then $\Delta_{\mathrm{BS}(k,l),\langle t\rangle}$ is well-defined.

\subsection*{Acknowledgements}
The authors thank Ashot Minasyan for comments on an earlier version of the manuscript, and  for pointing out a necessary correction in the proof of \Cref{thmx:last}. The authors also thank Anthony Genevois for comments and pointing us to his work with Tessera, see \Cref{ex:wreath}. 
The first author would like to thank his PhD supervisor Professor Ian Leary. The second author thanks Noel Brady for discussions on the topics of the article. The first author was supported by the Engineering and Physical Sciences Research Council grant number 2127970. The second author acknowledges funding by the Natural Sciences and Engineering Research Council of Canada, NSERC. The third author was supported by grant PAPIIT-IA101221.
All three authors would like to thank the organisers of the online seminar Algebra at Bicocca, without which this collaboration would not have happened. 
We thank the anonymous referee for a number of helpful comments.

\section{Combinatorial Dehn functions and fine graphs}
The goal of this section is to prove Theorem~\ref{thmx.relDehn2complex}.
We use the notion of disk diagram in a combinatorial complex, for definitions see for example~\cite{McW02}.
We begin by recalling the definition of a combinatorial Dehn function, then we prove each direction of \Cref{thmx.relDehn2complex} individually as \Cref{lem.relDehn.1.if} and \Cref{lem.relDehn.1.onlyif}.  Note that \Cref{lem.relDehn.1.onlyif} does not require the hypothesis of finite edge stabilisers.

Suppose $X$ is a combinatorial $2$-complex and let $c:S^1\to X$ be a closed path in $X^{(1)}$ that is null-homotopic in $X$. Then there is $D$ a disk diagram $i:D^2\to X$ \emph{spanning $c$}, that is, $i$ is a combinatorial map and $i(\partial D^2)=c$.  Let $\mathrm{Area}(D)$ denote the number of faces of $D$ and define \[\delta_X(c):=\min\{\mathrm{Area}(D)\colon D\text{ is a disk spanning } c\},\]
the \emph{combinatorial Dehn function $\Delta_X$ of $X$ is defined to be}
\[\Delta_X(n):=\max\{\delta_X(c)\colon c \text{ is a closed path in }X^{(1)}, \text{ null-homotopic in } X, \text{ with } |c|\leq n\}. \]

Unless otherwise stated all graphs in this article are assumed to be simplicial.  A \emph{circuit} in a simplicial graph is an embedded close path.  We recall the following definition due to Bowditch \cite[Proposition~2.1]{Bo12}.  A graph $\Gamma$ is \emph{fine} if for every $n\geq 0$ and every edge $e$ in $\Gamma$ there are finitely many circuits of length less than or equal to $n$ which contains $e$.

\begin{theorem}[Theorem~\ref{thmx.relDehn2complex}]\label{thm.relDehn2complex}
Let $X$ be a cocompact simply connected combinatorial $G$-$2$-complex with finite edge stabilisers.   The combinatorial Dehn function $\Delta_X$ of $X$ takes only finite values if and only if the $1$-skeleton of $X$ is a fine graph.
\end{theorem}

The next three lemmas prove the theorem.  The method is essentially a van Kampen diagram approach to the proof of \cite[Theorem~1.3]{MP16}.  The first lemma is a triviality.

\begin{lemma}
Let $X$ be a cocompact simply connected combinatorial $G$-$2$-complex with finite edge stabilisers, then each edge is contained in finitely many $2$-cells.
\end{lemma}

The next lemma proves the ``only if'' direction of Theorem~\ref{thm.relDehn2complex}.

\begin{lemma}\label{lem.relDehn.1.if}
Let $X$ be a cocompact simply connected combinatorial $G$-$2$-complex with finite edge stabilisers.  If the combinatorial Dehn function $\Delta_X$ of $X$ is well-defined then $X^{(1)}$ is a fine graph.
\end{lemma}
\begin{proof}

Let $D$ be a cellular $2$-disc.  We say $D$ is \emph{golden} if $D$ has a an enumeration of its $2$-cells $f_1,...,f_k$ with the property that $\partial f_{i+1}$ contains a $1$-cell of the subcomplex induced by $f_1\cup\dots\cup f_i$, and there is a cellular map $D\to X$.

Let $R$ be a $2$-cell and $f_1:R\to X$, then a simple counting argument yields there are only finitely many golden disks with at most $n\geq0$ faces making the following diagram commute:
\[\begin{tikzcd}
R \arrow[d, tail] \arrow[r, "f_1"] & X \\
D \arrow[ru]                              &  
\end{tikzcd} \]

Observe that by taking the minimal area filling for a circuit $c$ of length $n$ in $X$ gives rise to a golden disk $D$ with at most $\Delta_X(n)$ many $2$-cells.  Now, there are only finitely many $2$-cells containing a given edge $e$, so by the previous paragraph there are only finitely many golden discs $D$ containing $e$ with at most $\Delta_X(n)$ many $2$-cells.  In particular, for each $n\geq0$ there are only finitely many circuits in $X$ of length less than or equal $n$ containing $e$.  It follows that $X^{(1)}$ is a fine graph.
\end{proof}

The next lemma proves the ``if'' direction of Theorem~\ref{thm.relDehn2complex}.  Note that we can drop the hypothesis of finite edge stabilisers.

\begin{lemma}\label{lem.relDehn.1.onlyif}
Let $X$ be a cocompact simply connected combinatorial $G$-$2$-complex. If $X^{(1)}$ is a fine graph then the combinatorial Dehn function $\Delta_X$ of $X$ is well-defined.
\end{lemma}

\begin{proof}
Let $Y_n$ denote the set of circuits of length less than or equal to $n$ in $X$.  

\textbf{Claim:} \emph{$Y_n$ is a $G$-set with finitely many orbits.}

Let $\{e_1,\dots,e_r\}$ be edges representing the orbits of the $G$-action on $X^{(1)}$.  Every circuit in $X$ of length less than or equal to $n$ can be translated to contain some $e_i$, the claim now follows from fineness of $X^{(1)}$. $\blackdiamond$

Let $A_n$ be an upper bound for the area of a circuit of length less than or equal to $n$ in $X$, this is well-defined by the previous claim.  Let $\gamma$ be a closed path without backtracks in $X$, then $\gamma$ can be expressed as a concatenation of closed paths $\gamma_1\dots\gamma_k$, such that $1\leq k\leq\Len(\gamma)$, for $i=1,\dots,k$ we have $\Len(\gamma_i)\leq\Len(\gamma)$ and each $\gamma_i$ is a circuit.  Now, filling each $\gamma_i$ we have
\[\Area(\gamma)\leq \sum_{i=1}^k \Area(\gamma_i)\leq \sum_{i=1}^kA_{\Len(\gamma_i)}\leq kA_{\Len(\gamma)}\leq\Len(\gamma)A_{\Len(\gamma)}. \]
This yields a finite upper bound for $\Delta_X(\ell)$ and so we conclude that $\Delta_X$ is well-defined.
\end{proof}

\begin{remark}
One can define a combinatorial Dehn function of a $2$-complex using circuits instead of arbitrary closed paths. Let us call this function $\Delta_X^{\rm circ}$. In this case, 
\[ \Delta_X^{\rm circ} (n) \leq \Delta_X(n) \leq \overline{\Delta_X^{\rm circ}}(n) \]
where $\overline{\Delta_X^{\rm circ}}$ is the superadditive closure of $\Delta_X^{\rm circ}$. We do not know whether $\Delta_X^{\rm circ} (n)$ is equivalent to $\Delta_X(n)$. This resembles a conjecture of Mark Sapir of whether the Dehn function of a finitely presented group is equivalent to a superadditive function (see~\cite{GubaSapir}).
\end{remark}

\section{Coarse isoperimetric functions}\label{sec:coarse:isop.functions}
To prove quasi-isometry invariance we will use the less general version of $\epsilon$-fillings for graphs and $2$-complexes defined in \cite{Osin06}.  The original definition, set up for essentially arbitrary metric spaces, can be found in \cite[Chapter~III.H.2]{BridsonHaefligerBook}.  The main result of this section is Proposition~\ref{prop:deltaX:vs:fN} - a generalisation of a result of Osin \cite[Theorem~2.53]{Osin06} alluded to in the introduction.

Let $X$ be a $2$-complex.  A \emph{singular combinatorial loop} $c:S^1\to X$ is a combinatorial structure on $S^1$ and a continuous map such that for every open cell of $S^1$, either $f|_e$ is a homeomorphism onto an open cell of $X$, or else $f(e)$ is contained in the $0$-skeleton of $X$.

Let $c$ be a combinatorial cycle in $X$.  An \emph{$\epsilon$-filling} of $c$ is a pair $(P,\Phi)$ consisting of a triangulation $P$ of a $2$-disc $D^2$ and a singular combinatorial map $\Phi:P^{(1)}\to X^{(1)}$, such that $\Phi|_{S^1}=c$ and the image under $\Phi$ of each face of $P$ is a set of diameter at most $\epsilon$.  Define $|\Phi|$ to be the number of faces of $\Phi$ and
\[\Area_\epsilon(c):=\min\{|\Phi|\ \colon \ \phi \text{ an $\epsilon$-filling of $c$}\}. \]
The \emph{coarse isoperimetric function} of $X$ is then defined to be
\[f^X_\epsilon(\ell):=\sup\{\Area_\epsilon(c)\ \colon\ \Len(c)\leq \ell\}. \]


\begin{definition}\label{def:asymp:relation}
For two functions $f,g\colon \NN\to \NN$, we say that $f$ is \emph{asymptotically less than} $g$, and we write $f\preceq g$ if there exist constants $C,K,L\in \NN$ such that
\[f(n)\leq Cg(Kn)+Ln.\]
Further we say $f$ is \emph{asymptotically equivalent} to $g$, and write $f \asymp g$ if $f\preceq g$ and $g\preceq f$.
\end{definition}

\begin{prop}\label{prop:deltaX:vs:fN}
Let $X$ be a cocompact simply connected combinatorial $G$-$2$-complex.  If $\Delta_X$ takes only finite values, then for $N\in\NN$ large enough $f^X_N$ takes only finite values and $f^X_N\asymp \Delta_X$.
\end{prop}
\begin{proof}
Since $X$ is a cocompact $G$-$2$-complex there are only finitely many $G$-orbits of $2$-cells in $X$.  Let $\{D_1,\dots,D_n\}$ denote a representative set of orbits and let $N$ be an integer greater than the maximum diameter of each disc $D_i$ for $i=1,\dots,n$.

First, we will show $f^X_N\preceq\Delta_X$.
Let $c:S^1\to X$ be a singular combinatorial loop. 
Let $\Phi:D\to X$ be a disk diagram of minimal area that fills $c:S^1\to X$.  Barycentric subdividing $D$ twice to obtain $D''$ yields a simplicial disk such that the image of each face in $X$ has diameter less than $N$, i.e. $(D'',\Phi)$ is an $N$-filling of $c$.  It follows that $\Area_\epsilon(c)\leq 12N\Delta_X(\Len(c))$.  In particular, $f^X_N\preceq\Delta_X$.

It remains to show that $\Delta_X\preceq f^X_N$.  Consider an $N$-filling $(P,\Phi)$ of a combinatorial loop $c$ in $X^{(1)}$.  Considering $(P,\Phi)$ as a $3N$-filling we may assume that each $0$-cell of $P$ maps to a $0$-cell of $X$ and each $1$-cell of $P$ maps to an edge path in $X$ of length at most $N$.  Thus, after subdividing $P$ at most $N$ times we may assume that $\Phi$ is cellular on $P^{(1)}$.  For each $2$-cell of the subdivided $P$, its boundary map determines a cellular loop in $X$ with length bounded by $3N$.  Now, we fill each such loop with some disc diagram $D\to X$ to obtain a diagram for $c$ in $X$ which has area at most $\Delta_X(3N)f_{3N}^X(\Len(c))$.  In particular we conclude that $\Delta_X(\Len(c))\preceq\Delta_X(3N)f_{3N}^X(\Len(c))$. 
\end{proof}

A connected graph $\Gamma$ is \emph{fillable} if, when considering $\Gamma$ with the length metric obtained by regarding each edge as an segment  of length one, there is an integer $k$ such that the coarse isoperimetric function $f_k^\Gamma$ takes only finite values. 

\begin{proposition}\emph{\cite[Proposition~III.H.2.2]{BridsonHaefligerBook}}\label{Prop:Bridson:qi:invariance}
If $\Gamma$ and $\Gamma'$ are quasi-isometric connected graphs such that $\Gamma$ is fillable, then $\Gamma'$ is fillable  and $f_k^{\Gamma}\asymp f_k^{\Gamma'}$ for large enough $k$.
\end{proposition}

\begin{remark}\label{rem.fillable}
If a connected graph $\Gamma$ is fillable, then there is a positive integer $m$ such that the complex obtained by attaching 2-cells to all circuits of length less than or equal to $m$ is simply connected.
\end{remark}

\section{Relative Dehn functions of groups}

\begin{definition}[Finite relative presentation]
Let $G$ be a group, $\calp$ an arbitrary collection of subgroups of $G$, and let $S$ be a subset of $G$. We say that $G$ is \emph{generated by $S$ relative to $\calp$} if $G$ is generated by the set $\cals=S\sqcup \bigsqcup_{P\in\calp}(P-\{1\})$, equivalently, the natural homomorphism
\begin{equation}\label{eq:rel:presentation}
    F=F(S)\ast \bigast_{P\in \calp}P\to G
\end{equation}
is surjective. In the case that $S$ is finite, $G$ is relatively finitely generated with respect to $\calp$. 

Let $R\subseteq F$ be a set that normally generates the kernel of the above homomorphism, then we say
\begin{equation}\label{G:relative:presentation}
    G=\langle S,\calp\ |\ R \rangle
\end{equation}
is a \emph{ presentation of $G$ relative to $\calp$}. If both $S$ and $R$ are finite we say $G$ is \emph{relatively finitely presented with respect to $\calp$}, or just, \emph{relatively finitely presented} if the collection $\calp$ is clear from the context, and \eqref{eq:rel:presentation} is a \emph{relative finite presentation}.
\end{definition}

\begin{definition}[Relative Dehn function of a relative presentation]
Let $G=\langle S,\calp\ |\ R \rangle$ be a relative presentation. For a word $W$ over  the alphabet $\cals=S\sqcup \bigsqcup_{P\in\calp}(P-\{1\})$ representing the trivial element in $G$, there is an expression
\begin{equation}\label{eq:relation}W=\prod_{i=1}^k f_i^{-1}R_i f_i\end{equation}
where $R_i\in R$ and $f_i\in F$.

We say a function $f\colon \NN\to \NN$ is a \emph{relative isoperimetric function} of the presentation $G=\langle S,\calp\ |\ R \rangle$ if, for any $n\in \NN$, and any word $W$ as above of length $\leq n$, one can write $W$ as in \eqref{eq:relation} with $k\leq f(n)$. The smallest relative isoperimetric function of $G=\langle S,\calp\ |\ R \rangle$ is called the \emph{relative Dehn function of $G$ with respect to $\calp$}, and it is denoted $\Delta_{G,\calp}$.
\end{definition}

\Cref{def:asymp:relation} and \Cref{Thm:Osin:well:defined} below justify the notation $\Delta_{G,\calp}$ for the relative Dehn function of $G$ with respect to $\calp$.


\begin{theorem} \emph{\cite[Theorem 2.34]{Osin06}} \label{Thm:Osin:well:defined} 
Let $G$ be a finitely presented group relative to $\calp$. Let $\Delta_1$ and $\Delta_2$ be the relative Dehn functions associated to two finite relative presentations. If $\Delta_1$ takes only finite values, then $\Delta_2$ takes only finite values, and $\Delta_1\asymp \Delta_2$.
\end{theorem}

\begin{definition}
[Osin-Cayley graph and Osin-Cayley complex] Assume $G$ has a relative presentation as in \eqref{G:relative:presentation}. We call the Cayley graph $\Gamma(G,\cals)$ with $\cals=S\sqcup \bigsqcup_{P\in\calp}(P-\{1\})$ the \emph{Osin-Cayley graph} and we denote it by $\bar\Gamma(G,\calp,S)$.  Note that in general this graph is not simplicial.

For each $P\in\calp$, denote by $R_P$ the set of all words in the alphabet $P-\{1\}$ that represent the identity in $P$, that is, we have the presentation $P=\langle P-\{1\}\ |\ R_P\rangle$. Also we have the following presentation
\[F=\langle S, \bigsqcup_{P\in \calp}(P-\{1\})\ |\ \bigsqcup_{P\in\calp} R_P \rangle.\]

The Osin-Cayley complex $\bar X(G,\calp,S)$ is the 2-complex with 1-skeleton $\bar\Gamma(G,\calp,S)$ and  we attach:
\begin{itemize}
    \item One 2-cell for each loop labelled with a word in $R$, which we call from now on \emph{$R$-cells}.
    \item One 2-cell for each loop labelled by a word in $\bigsqcup_{P\in\calp} R_P$, which we call from now on \emph{$\calp$-cells}.
\end{itemize}
\end{definition}

\begin{remark}\label{rem:def_of_rel_D_f}
By \cite[Definition~2.31]{Osin06} the relative Dehn function $\Delta_{G,\calp}$ can be described as follows. For any combinatorial loop $\gamma\colon S^1\to \bar X(G,\calp,S)$, the relative area $\Area^{rel}(\gamma)$ of $\gamma$ is the number of $R$-cells in a minimal disk diagram for $\gamma$, where minimality is with respect to the number of $R$-cells. Then 
\[\Delta_{G,\calp}(n)=\max\{\Area^{rel}(\gamma)\colon\gamma \text{ is a loop in $\bar X(G,\calp,S)$  of length at most }n\}.\]
\end{remark}

\begin{definition}[Coned-off Cayley graph]\label{def:conedOffCayleyGraph}
Let $G$ be a group, let $\calp$ be an arbitrary collection of subgroups of $G$, and let $S$ be a generating set of $G$. Denote by $G/\calp$ the set of all cosets $gP$ with $g\in G$ and $P\in\calP$. The \emph{coned-off Cayley graph  of $G$ with respect to $\calp$} is the graph $\hat\Gamma(G,\calp,S)$ with vertex set $G\cup G/\calp$ and edges are of the following type
\begin{itemize}
    \item $\{g,gs\}$ for $s\in S$,
    \item $\{x, gP\}$ for $g\in G$, $P\in \calp$ and $x\in gP$.
\end{itemize}
We call vertices of the form $gP$ \emph{cone points}. 
\end{definition}

Note that $\hat\Gamma(G,\calp,S)$ contains the Cayley graph of  $G$ with respect to the generating set $S$, and the quasi-isometry type of 
$\hat\Gamma(G,\calp,S)$ is independent of the finite generating set $S$ of $G$. This justifies the notation $\hat\Gamma(G,\calp)$ that we use throught the article.

\begin{definition}[A natural quasi-isometry between $\bar\Gamma(G,\calp,S)$ and $\hat\Gamma(G,\calp,S)$]\label{def:varphi} Assume $G$ is generated by $S$ relatively to $\calp$. Let \[\varphi\colon \bar\Gamma(G,\calp,S)\to\hat\Gamma(G,\calp,S)\] be the map defined as follows.
Add a vertex at the midpoint of each edge  $e=\{g,gh\}$ of $\bar\Gamma(G,\calp, S)$  with $h\in P$, $p\in\calp$, and label in $P$. Consider the inclusion of the vertex set of $\bar\Gamma(G, \calp,S)$ into the vertex set of $\hat\Gamma=\hat \Gamma(G,\calp,S)$.  Observe that this map extends to a $G$-equivariant cellular map between  $\bar\Gamma(G,\calp, S)$ and $\hat \Gamma(G,\calp,S)$. Specifically, for an edge $e=\{g,gh\}$ with $h\in P$ and label in $P$ of $\bar\Gamma(G,\calp,S)$, the midpoint of $e$ maps to the vertex $gP$;  an edge $\{g,gs\}$ with label in $S$ is an edge that is  common to both  $\bar\Gamma(G, \calp,S)$ and $\hat \Gamma(G,\calp,S)$.
Observe that the map $\varphi\colon\bar\Gamma(G,\calp,S)\to \hat \Gamma(G,\calP,S)$ is indeed a $(1,1)$-quasi-isometry.
\end{definition}

\begin{proposition}\label{prop:Osin}
Let $G$ be a group and let $\calp$ be a collection of subgroups. If $\Delta_{G,\calp}$ is well-defined, then $\Delta_{G,\calp} \asymp f^{\hat\Gamma(G,\calp)}_N$ for all sufficiently large integers $N$.
\end{proposition}
\begin{proof}
This is a re-statement of Osin's result~\cite[Theorem 2.53]{Osin06} modulo the fact that $\hat\Gamma(G,\calp, S)$ and the Cayley graph $\bar\Gamma(G,\calp,S)$ are quasi-isometric graphs, see \Cref{def:varphi}. 
\end{proof}

\begin{proposition}\label{prop:Referee}
Let $G$ be a group, $\calp$ be a collection of subgroups, and $S$ a  relative generating set.  Let $\tilde G$ be the free product $F(S)\ast \bigast_{P\in \calp}  P$,  and consider the short exact sequence, 
\[1 \to N \hookrightarrow \tilde G \xrightarrow{\varphi} G \to 1 \]
where $\varphi$ is  the homomorphism induced by the inclusion $S\cup \bigcup_{i=1}^n P_i$ into $G$, and $N$ is the kernel of $\varphi$. Then the coned-off Cayley graph $\hat\Gamma=\hat\Gamma(G,\calp, S)$ is connected and for any vertex $x_0$ of $\hat\Gamma$,  there is a group isomorphism
\[ N\to \pi_1(\hat\Gamma, x_0),\qquad g\mapsto [\gamma_g],\]
where $\gamma_g$ is an combinatorial closed path in $\hat\Gamma$ based at $x_0$.
\end{proposition}
\begin{proof}
Consider the splitting of $\tilde G$ as the fundamental group of the graph of groups $\mathbf{Y}$ that consists of a vertex $v$ labelled with the trivial group, one vertex $v_P$ labelled with each $P\in \calP$ respectively, one edge $e_P$ that joins $v$ with $v_P$ for each $P\in \calp$ labelled with the trivial group, and one edge loop $e_s$ based at $v$ for each $s\in S$ labelled with the trivial group. 

Let $\calt$ be the Bass-Serre tree of $\mathbf{Y}$, see \cite{SerreTrees}. 
Since each subgroup $P$ of $\tilde G$ survives in the quotient $G$, we have that the subgroup $N$ acts freely on $\calt$, and the quotient map $\rho\colon \calt\to \calt/N$ is a covering map. Moreover,  $G$ acts on the quotient $\calt/N$. We leave the reader to verify that the quotient  $\calt/N$ is $G$-homeomorphic to the coned-off Cayley graph $\hat\Gamma(G,\calp,S)$. 

Fix a vertex $\tilde x_0$  of $\calt$ such that $\rho(\tilde x_0)=x_0$.  Then any element $g$ of $N$ induces a unique embedded path $\alpha_g$ from $\tilde x_0$ to $g\tilde x_0$. Let $\gamma_g=\rho\circ \alpha_g$ and note it is a closed combinatorial path in $\hat\Gamma$ based at $x_0$.
Since $\calt$ is simply connected,
standard covering space theory implies that the map
$ N\to \pi_1(\hat\Gamma, \rho(x_0))$ given by  $g\mapsto [\gamma_g]$ is a group isomorphism.
 \end{proof}

\begin{definition}[Coned-off Cayley complex $\hat X(G,\calp,S)$]
Consider a finite relative presentation $G=\langle S,\calp\ |\ R \rangle$. The \emph{coned-off Cayley complex $\hat X(G,\calp,S)$} of $G$ is a $2$-dimensional $G$-complex  with $1$-skeleton the coned-off Cayley graph $\hat\Gamma(G,\calp,S)$ defined as follows. 

We use the setup of Proposition~\ref{prop:Referee}. In particular, $N$ is the  normal subgroup of $\tilde G$ generated by $R$, we have fix a vertex $x_0$ of $\hat\Gamma$,
and we have a group isomorphism $N\to \pi_1(\hat\Gamma,x_0)$ given by $g\mapsto [\gamma_g]$ where $\gamma_g$ is a combinatorial closed path based at $x_0$.

For $g\in G$ and $r\in R$,  let $g\cdot\gamma_r$ be the translated closed path in $\hat\Gamma$ without an initial point, i.e., these are cellular maps from $S^1 \to \hat\Gamma$. Consider the $G$-set $\Omega=\{g.\gamma_{r}\mid r\in R, g\in G\}$ of closed paths in $\hat\Gamma$. The complex $\hat X$  is then  obtained by attaching a $2$-cell to $\hat \Gamma$ for every closed path in $\Omega$. In particular, the pointwise $G$-stabilizer  of a $2$-cell of $\hat X$ coincides with the pointwise $G$-stabilizer of its boundary path. The natural isomorphism from $N$ to $\pi_1(\hat\Gamma, \rho(x_0))$ implies that $\hat X$ is simply connected. Moreover,  the $G$-action is cocompact since $R$ is finite.
\end{definition} 

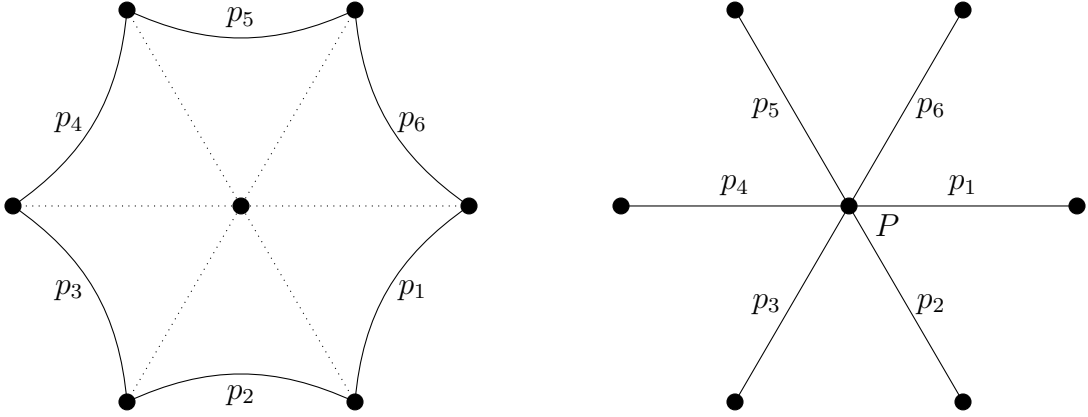
\begin{figure}[ht]
    \centering
    \begin{tikzpicture}[vert/.style={circle, draw=black, fill=black, thin, minimum size=0.1, scale=0.5}, >=stealth, scale=1]
\tkzDefPoint(0,0){O}\tkzDefPoint(3,0){A}
\tkzDefPointsBy[rotation=center O angle -360/6](A,B,C,D,E,F){B,C,D,E,F}
\tkzDrawPoints[fill =black,size=6,color=black](O,A,B,C,D,E,F)
\draw (A) to[bend left =-25] node[midway, right] {$p_1$} (B) {};
\draw (B) to[bend left =-25] node[midway, below, yshift=0mm] {$p_2$} (C) {};
\draw (C) to[bend left =-25] node[midway, left,  xshift=0mm] {$p_3$} (D) {};
\draw (D) to[bend left =-25] node[midway, left] {$p_4$} (E) {};
\draw (E) to[bend left =-25] node[midway, above] {$p_5$} (F) {};
\draw (F) to[bend left =-25] node[midway, right] {$p_6$} (A) {};
\draw[dotted] (A) to (O);
\draw[dotted] (B) to (O);
\draw[dotted] (C) to (O);
\draw[dotted] (D) to (O);
\draw[dotted] (E) to (O);
\draw[dotted] (F) to (O);

\tkzDefPoint(8,0){OO}\tkzDefPoint(11,0){A2}
\tkzDefPointsBy[rotation=center OO angle -360/6](A2,B2,C2,D2,E2,F2){B2,C2,D2,E2,F2}
\tkzDrawPoints[fill =black,size=6,color=black](OO,A2,B2,C2,D2,E2,F2)
\draw (A2) to node[midway, above] {$p_1$} (OO);
\draw (B2) to node[midway, right] {$p_2$} (OO);
\draw (C2) to node[midway, left] {$p_3$} (OO);
\draw (D2) to node[midway, above] {$p_4$} (OO);
\draw (E2) to node[midway, left] {$p_5$} (OO);
\draw (F2) to node[midway, right] {$p_6$} (OO);
\node at (8.5,-0.25) {$P$};
\end{tikzpicture}
    \caption{The image of the boundary of a $\calP$-cell on $\bar\Gamma(G,\calp,S)$ under the quasi-isometry $\varphi$. }
    \label{fig:disctype1.OsinToConed}
\end{figure}

\begin{definition}[A natural map between $\bar X(G,\calp,S)$ and $\hat X(G,\calp,S)$]  There exists a $G$-map  $\varphi\colon\bar X(G,\calp,S)\to \hat X(G,\calp,S)$ that extends the natural quasi-isometry $\varphi\colon \bar\Gamma(G,\calp,S)\to \hat\Gamma(G,\calp,S)$.  In particular, we have a commutative diagram 
\[\begin{tikzcd}
\bar\Gamma(G,\calp,S) \arrow[r, "\varphi"] \arrow[hookrightarrow]{d} & \hat \Gamma(G,\calp,S) \arrow[hookrightarrow]{d} \\ \bar X(G,\calp,S)   \arrow[r, "\varphi"] &  \hat X(G,\calp,S).
\end{tikzcd}\]
Specifically every $R$-cell in $\bar X(G,\calp,S)$  is sent homeomorphically to the corresponding 2-cell in  $\hat X(G,\calp,S)$, while every $\calp$-cell in $\bar X(G,\calp,S)$ is collapsed to a star-like 1-complex as we see in  \Cref{fig:disctype1.OsinToConed}.
\end{definition}

\begin{remark}\label{rem:diagrams}
The following statements are straightforward to verify from the definition of $\varphi\colon \bar X(G,\calp,S) \to \hat X(G,\calp,S)$ and \Cref{fig:disctype1.OsinToConed}. Denote by $\Delta_{\hat X}$ the combinatorial Dehn function of $\hat X(G,\calp,S)$.

\begin{enumerate}
    \item Let $\hat\gamma \colon S^1\to \hat X(G,\calp,S)$ be a loop   with no backtracks in the coned-off Cayley complex. Then we can pull-back $\hat\gamma$ to a loop  $\gamma \colon S^1\to \bar X(G,\calp,S)$ in such a way that the following diagram commutes
\[
\xymatrix{
& S^1\ar@{-->}[dl]_{\exists!\gamma} \ar[dr]^{\hat\gamma} & \\
 \bar\Gamma(G,\calp,S)\ar[rr]^\varphi& &\hat \Gamma(G,\calp,S).
}\]
 Let $D\to \bar X(G,\calp,S)$ be a disk diagram filling a combinatorial loop $\gamma:S^1\to \bar X(G,\calp,S)$. Then there exists a disk diagram $\hat D\to \hat X(G,\calp,S)$ so that the following diagram commutes
 \[
\xymatrix{
 S^1 \ar@{=}[d] \ar^\gamma[r]&\bar X(G,\calp,S)\ar[r]^{\varphi}  & \hat X(G,\calp,S) \\
 \partial D \ar[r]& D \ar@{-->}[r] \ar[u]& \hat D. \ar@{-->}[u]}
 \]

 \item Let $\gamma \colon S^1\to \bar X(G,\calp,S)$ be a combinatorial loop of length $n$, then we can push it to a loop $\hat\gamma=\varphi\circ\gamma \colon S^1\to \hat X(G,\calp,S)$ of length at most $2n$, that is, we have the following commutative diagram
\[\xymatrix{
 & S^1\ar[dl]_\gamma \ar@{-->}[dr]^{\varphi\circ \gamma} & \\
 \bar\Gamma(G,\calp,S)\ar[rr]^\varphi& &\hat \Gamma(G,\calp,S)}.\]
 Let $\hat D\to\hat X(G,\calp,S)$ be a  disk diagram filling the cycle $\hat\gamma=\varphi\circ \gamma$.  Then there exists a disk diagram $D\to \bar X$ such that the following diagram commutes
 \[
\xymatrix{
 \bar X(G,\calp,S)\ar[r]^{\varphi}  & \hat X(G,\calp,S) & S^1 \ar@{=}[d] \ar[l]_{\hat\gamma} \\
  D \ar@{-->}[r] \ar@{-->}[u]& \hat D \ar[u] & \partial \hat D. \ar@{^(->}[l]}
 \]
 \item In both items above, $\Area^{rel}(D)=\Area(\hat D)$.
\end{enumerate}
\end{remark}

\begin{proposition}\label{lem:dehnfunction:connedoff:cayley}
Let $G=\langle S,\calp\ |\ R \rangle$ be a finite relative presentation, and let $\Delta_{G,\calp}$ and  $\hat X$ be the corresponding  relative Dehn function and coned-off Cayley complex respectively. Then $\Delta_{G,\calp}(n)\asymp\Delta_{\hat X}(n)$ for every $n\in\NN$.   
\end{proposition}
\begin{proof}
 Let $\hat\gamma \colon S^1\to \hat X(G,\calp,S)$ be a loop of length $n$ with no backtracks in the coned-off Cayley complex. By the first item of Remark~\ref{rem:diagrams} and considering a minimal relative area disk diagram $D\to \bar X(G,\calp,S)$ filling a pull-back cycle $\hat\gamma\colon S^1 \to \hat\Gamma(G,\calp,S)$ of $\gamma$, it follows that
 \[ \Delta_{G,\calp}(|\hat \gamma|) \geq \Delta_{G,\calp}(|\gamma|) \geq  \Area^{rel}(D) = \Area(\hat D) \geq \Area (\hat\gamma ) \]
 where the equality comes from the third item of \Cref{rem:diagrams}.
Therefore $\Delta_{G,\calp}(n)\geq \Delta_{\hat X}(n)$ for all $n\in\NN$.
Analogously, let $\gamma\colon S^1\to \bar X(G,\calp,S)$ be a combinatorial loop.
By the second item of Remark~\ref{rem:diagrams} and considering a minimal area disk diagram
$\hat D \to \hat X(G,\calp,S)$ filling  $\hat\gamma=\varphi\circ \gamma$, it follows that
 \[ \Area^{rel}(\gamma) \leq \Area^{rel}(D) = \Area(\hat D) \leq \Delta_{\hat X}(|\hat\gamma|) \leq \Delta_{\hat X}(2|\gamma|)  \]
 and hence $\Delta_{G,\calp}(n)\leq \Delta_{\hat X}(2n)$ for all $n\in\NN$.
 \end{proof}

The following corollary is a direct consequence of \Cref{thm.relDehn2complex} and \Cref{lem:dehnfunction:connedoff:cayley}.

\begin{corollary}\label{cor:DeltaFine}
Let $G$ be finitely presented relative to a collection of subgroups $\calp$. The following statements are equivalent:
\begin{enumerate}
    \item  The relative Dehn function $\Delta_{G,\calp}$ takes only finite values.
    \item The graph $\hat\Gamma(G,\calp)$ is fine.
\end{enumerate} 
\end{corollary}

In the proof of  \cite[Proposition~2.50]{GM08} is implicit that $(1)$ implies $(2)$ of the previous theorem.

The following corollary is a straightforward consequence of \Cref{lem:dehnfunction:connedoff:cayley} and \Cref{prop:deltaX:vs:fN}.

\begin{corollary}\label{lem:fillable01}
Let $G$ be a group finitely presented relative to a finite collection of subgroups $\calp$. If $\Delta_{G,\calp}$ takes only finite values, then $\hat\Gamma(G,\calp)$ is fillable for some integer $m$. 
\end{corollary}


\begin{proposition}\label{prop:fillable03} Let $G$ be a group finitely generated by $S$ with respect to $\calp$.
If $\hat\Gamma(G,\calp,S)$ is connected, fine, cocompact, and $k$-fillable,   then $G$ is finitely presented relative to $\calp$.
\end{proposition}
\begin{proof}
We use the setup of Proposition~\ref{prop:Referee}. In particular, $N$ is the  normal subgroup of $\tilde G$ generated by $R$, we  fix  a vertex $x_0$ of $\hat\Gamma=\hat\Gamma(G,\calp, S)$,
and we have a group isomorphism $\psi\colon N\to \pi_1(\hat\Gamma,x_0)$.

Since $\hat\Gamma$ is $k$-fillable, there is an integer $m$ such that the complex $\hat X$ obtained by attaching $2$-cells with boundary paths the circuits of length at most $m$ is simply connected, see Remark~\ref{rem.fillable}.

Since $\hat\Gamma$ is fine and  there are finitely many $G$-orbits of edges,  there are finitely many $G$-orbits of circuits of length at most $m$. Let $\{\gamma_1, \ldots , \gamma_\ell\}$ be a collection of representatives of circuits of length $m$, and after translations assume that each $\gamma_i$ contains the vertex $x_0$ corresponding to the identity element of $G$. Then each $\gamma_i$ defines an element of the fundamental group $\pi_1(\hat\Gamma, x_0)$.    Let $r_i \in N$ be defined by  $\psi(r_i) = \gamma_i$.

Since $\hat X$ is simply connected, we have that $\pi_1(\hat\Gamma, x_0)$ is generated by the closed paths arising as concatenations of the form $\alpha_g\cdot  \gamma_i \cdot \bar \alpha_g$ for $g\in \tilde G$,  where $\alpha_g$ is the projection via $\rho$ of the unique path from $\tilde x_0$ to $g.\tilde x_0$. Equivalently, $N$ is generated by the elements $gr_ig^{-1}$ for $g\in \tilde G$. We have shown that $N$ is normally generated by $\calr=\{r_1,\ldots , r_\ell\}$.

Since  $\hat\Gamma=\hat\Gamma(G,\calp)$ is cocompact, the collection $\calp$ is finite. Therefore $\langle S, \calp \ | \calr \rangle$ is a finite relative presentation of $G$. 
\end{proof}

Summarising the results of this section we obtain \Cref{thm:ConefOff} below.  

\begin{thm}[\Cref{thmx:ConefOff}]
\label{thm:ConefOff}
Let $G$ be a group  finitely generated relative to a finite collection of subgroups $\calp$. If $G$ is finitely presented relative to $\calp$, then
\begin{enumerate}
    \item $\hat\Gamma(G,\calp)$ is fillable.
    \item The relative Dehn function $\Delta_{G,\calp}$ is well-defined if and only if  $\hat\Gamma(G,\calp)$ is fine graph.
\end{enumerate}
Conversely, if $\hat\Gamma(G,\calp)$ is   fine and fillable, then
$G$ is finitely presented relative to $\calp$  and hence $\Delta_{G,\calp}$ is well-defined.
\end{thm}
\begin{proof}
This follows from \Cref{cor:DeltaFine}, \Cref{lem:fillable01} and \Cref{prop:fillable03}.
\end{proof}

Note that \Cref{thmx:ConefOff} from the introduction is a particular case \Cref{thm:ConefOff}. 

\section{Fineness and quasi-isometries of pairs}
In this section we will prove Theorem~\ref{thmx:fine} from the introduction.  The heart of the argument is establishing \Cref{prop:FineQI-2} which gives conditions on a quasi-isometry of pairs $q\colon(G,\calp)\to(H,\calq)$ to induce a quasi-isometry of coned off Cayley graphs.  The remainder of the section then works towards replacing the geometric-set-theoretic conditions on $q$ with algebraic conditions on $\calp$ and $\calq$.  This yields Proposition~\ref{prop:DotqIsFunction}.  Finally, we give a proof of Theorem~\ref{thmx:fine}.

Another equivalent definition of Bowditch's fine graphs is used in this section~\cite[Proposition 2.1]{Bo12}.

\begin{definition}[Fine]\label{def:fine2}
Let $\Gamma$ be a graph and let $v$ be a vertex of  $\Gamma$. Let \begin{align*}
T_v \Gamma = \{w \in V(\Gamma) \mid \{v,w\}\in E(\Gamma)\}.
\end{align*}
denote the set of the  vertices adjacent to $v$.
For $x,y \in T_v \Gamma$,  the \emph{angle metric} $\angle_v (x,y)$ is the length of the shortest path in the graph  $\Gamma \setminus \{v\}$ between $x$ and $y$, with $\angle_v (x,y) = \infty$ if there is no such path.  The graph $\Gamma$ is \emph{fine at $v$} if $(T_v \Gamma,\angle_v)$ is a locally finite metric space. The graph $\Gamma$ is \emph{fine at    $C \subseteq V(\Gamma)$} if $\Gamma$ is fine at $v$  for all $v \in C$. The graph $\Gamma$ is a \emph{fine graph} if it is fine at every vertex.
\end{definition}

\begin{definition}[Quasi-isometry of Pairs]\label{def:qipairs}
Consider two pairs $(G, \mathcal{P})$ and $(H, \mathcal{Q})$ where $G$ and $H$ are finitely generated groups with chosen word metrics $\dist_G$ and $\dist_H$ with respect to some finite generating sets. Denote the  Hausdorff distance between subsets of $H$ by $\Hdist_H$.
An $(L,C)$-quasi-isometry $q\colon G \to H$ is an \emph{$(L,C,M)$-quasi-isometry of  pairs}  $q\colon (G, \mathcal{P})\to (H, \mathcal{Q})$ if the relation
\begin{equation*}\label{eq:def-qi-relation} 
\dot{q} = \{ (A, B)\in G/\mathcal{P} \times H/\mathcal{Q} \colon \Hdist_H(q(A), B) <M  \} \end{equation*}
satisfies that the projections into $G/\mathcal{P}$ and $H/\mathcal{Q}$ are surjective. 
\end{definition}

In this section we explicitly use the relational approach of the notion of a function between sets:  a function $f$ from $A$ to $B$ is a subset of $A\times B$ so that for every $a\in A$ there is a unique $b\in B$ such that $(a,b)\in f$. 

The proof of \Cref{thmx:fine}, the main objective of this section,  relies on the study of the relation $\dot q$ defined by a quasi-isometry of pairs $q\colon (G,\calp)\to(H,\calq)$. We will show that in the case that $\dot q$ defines a bijection $G/\calp \to H/\calq$,  the coned-off Cayley graphs $\hat\Gamma(G,\calp)$ and $\hat\Gamma(H,\calq)$ share global an local geometric conditions, see Proposition~\ref{prop:FineQI-2}. In the second part of the section, we provide algebraic  conditions guaranteeing that the relation $\dot q$ is a bijection, see Proposition~\ref{prop:DotqIsFunction}. 

\begin{remark}
Note that in \Cref{def:qipairs} the notion of a quasi-isometry of pairs is independent of the chosen finite generating sets for $G$ and $H$.  In the case where we want to keep track of specific generating sets we use the following notation. If $G$ and $H$ are groups generated by finite generating sets $S_0$ and $T_0$ respectively, by a quasi-isometry of pairs $(G, \mathcal{P},S_0)\to (H, \mathcal{Q},T_0)$ we mean a quasi-isometry of pairs $(G,\calp)\to(H,\calq)$ with respect to the word metrics induced by $S_0$ and $T_0$.  
\end{remark}

\begin{remark}
If $\calp$ is a finite collection, then the metric space $(G/\calp, \Hdist)$ is locally finite. Indeed, fixing $P\in\calp$ and $r>0$, there are finitely many left cosets in $G/\calp$ such that $\Hdist(P,gP)<r$. Moreover, the left $G$-action on $G/\mathcal{P}$ by multiplication on the left  preserves the Hausdorff distance $\Hdist$ between subsets of $G$ and hence it is an action by isometries. \end{remark}

\begin{remark}
If $q\colon (G,\calp)\to (H,\calq)$ is an $(L,C,M)$-quasi-isometry of pairs, and  $\dot{q}$ is a function $G/\calp \to H/\calq$, then
\[ \frac{1}{L}\Hdist(A,B)-C-2M \leq \Hdist(\dot{q}(A),\dot{q}(B))\leq L\Hdist(A,B)+C+2M.\]
In particular, $\dot{q}\colon(G/\calp, \Hdist) \to (H/\calq, \Hdist)$ is a quasi-isometry. 
\end{remark}

The main technical result of this section is the following proposition. Note that given a connected graph $\Gamma$ we consider the vertex set as a metric space with metric induced by the path metric. In particular, a quasi-isometry between graphs is a function of the vertex sets satisfying the usual axioms.

 \begin{proposition}\label{prop:FineQI-2}
Let  $G$ and $H$ be groups, let $S\subset G$ and $T\subset H$, and let $S_0\subset S$ and $T_0\subset T$ be finite generating sets of $G$ and $H$ respectively. Consider collections $\calP$ and $\mathcal Q$ of subgroups of $G$ and $H$ respectively. Let  $q\colon G\to H$ be a function.

Suppose $q$ is a quasi-isometry $\Gamma(G,S) \to \Gamma(H,T)$, is a 
  quasi-isometry of pairs $(G, \mathcal{P},S_0)\to (H, \mathcal{Q},T_0)$, and 
   $\dot{q}$ is a bijection $G/\calp \to H/\calq$. 
   \begin{enumerate}
   \item If $\hat q = q\cup \dot{q}$, then $\hat q$ is a   quasi-isometry   $\hat \Gamma (G,\calp, S) \to \hat   \Gamma (H,\calq, T)$.
   \item If $\hat \Gamma (H,\calq, T)$ is fine at cone vertices, then $\hat \Gamma (G,\calp, S)$ is fine at cone vertices.
   \end{enumerate}
 \end{proposition}
 
  \begin{remark}
 There are algebraic conditions on $\calp$ and $\calq$ that imply that $\dot{q}$ is a bijection, see Proposition~\ref{prop:DotqIsFunction}.
\end{remark}

 \begin{corollary}\label{cor:FineQI}
Suppose that $q\colon (G, \calp ) \to (H, \calq)$ is an quasi-isometry of pairs and $\dot{q}$ is a bijection. Then $\hat q\colon \hat \Gamma (G,\calp) \to \hat \Gamma (H,\calq)$ is a quasi-isometry, and if $\hat \Gamma (H,\calq)$ is a fine graph, then $\hat \Gamma (G,\calp)$ is a fine graph.
\end{corollary}

The following argument is patterned from~\cite[Proof of Proposition 5.4]{MPR2021}.

\begin{proof}[Proof of Proposition~\ref{prop:FineQI-2}]
Suppose $q\colon \Gamma(G,S) \to \Gamma(H,T)$ is a $(\bar L, \bar C)$-quasi-isometry and 
 $q\colon (G, \mathcal{P},S_0)\to (H, \mathcal{Q},T_0)$ is a
 $(L,C,M)$-quasi-isometry of pairs.
 
 For any path   $\alpha=[v_0,v_1,\ldots ,v_\ell]$  in $\hat \Gamma (G, \calp, S)$, let $\hat q(\alpha)$ denote a path in $\hat \Gamma (H, \calq, T)$ from $\hat q(v_0)$ to $\hat q(v_\ell)$ obtained as the concatenation of paths $\beta_0, \ldots , \beta_{\ell-1}$ where $\beta_i$ is a path from $\hat q(v_i)$ to $\hat q(v_{i+1})$ defined as follows:
\begin{enumerate}
    \item If $v_i$ and $v_{i+1}$ are elements of $G$, then $\beta_i$ is a geodesic in $\Gamma(H,T)$ from $q(v_i)$ to $q(v_{i+1})$. Since $q\colon \Gamma(G,S)\to \Gamma(H,T)$ is a  $(\bar L, \bar C)$-quasi-isometry,   $\beta_i$ has length bounded by $\bar L +\bar C$.
    
    \item Suppose $v_i\in G$ and $v_{i+1}\in G/\calp$. Observe that $v_i$ is an element of the left coset $v_{i+1}$. Since $q\colon (G,\calp, S_0) \to (H,\calq, T_0)$ is an $(L,C,M)$-quasi-isometry of pairs, there is a geodesic of length at most $M$ in $\Gamma(H,T_0)$ from $q(v_i)$ to an element $w$ of the left coset $\dot{q}(v_{i+1})$. Let $\beta_i$ be the  concatenation of this geodesic in $\Gamma (H,T_0)$ followed by the edge between $w$ and the cone vertex $\dot{q}(v_{i+1})$.  
    Observe that $\beta_i$ is a path of length at most $M+1$ in $\hat \Gamma (H, \calq, T)$.
    
    \item If $v_i\in G/\calp$ and $v_{i+1} \in G$ then $\beta_i$ is defined in an analogous way as in the previous case, and also has length at most $M+1$.
\end{enumerate}
Observe that
\[ |\hat q(\alpha)| \leq (\bar L+\bar C+M+ 1) |\alpha|.\]
The above inequality applied in the case that $\alpha$ is a geodesic between vertices $x$ and $y$ of $\hat\Gamma(G,\calp,S)$ implies that
\[ \dist_{\hat \Gamma(H,\calq, T)}(\hat q(x), \hat q(y)) \leq (\bar L+ \bar C+M+1)   \dist_{\hat \Gamma (G,\calp,S)} (x,y) \]
for any pair of vertices $x,y$ of $\hat\Gamma(G,\calp,S)$. By symmetry an analogous inequality holds for vertices of $\hat\Gamma(H,\calq)$. 
Since $\dot q$ is a bijection, the definition of $\hat q(\alpha)$ shows that $\alpha$   passes through a cone vertex $A$ if and only if $\hat q(\alpha)$  passes through the cone vertex $\dot q(A)$. We summarise this discussion in the following lemma. 
 
\begin{lemma}\label{lem:QI-HatGamma-2}
There are constants $\hat L\geq 1$ and $\hat C\geq 0$ such that:
  \begin{enumerate}
      \item The function $\hat q$ is a $(\hat L, \hat C)$-quasi-isometry from $\hat \Gamma (G, \calp, S)$ to $\hat \Gamma (H, \calq, T)$.
      \item Let $\alpha$ be a path in $\hat \Gamma(G, \calp, S)$.
      \begin{enumerate}
          \item For any $A\in G/\calp$,   $\alpha$  passes through the cone  vertex $A$ if and only if  $\hat q(\alpha)$ passes through the cone vertex $\dot{q}(A)$.
        \item $|\hat q(\alpha)| \leq \hat L\    |\alpha|$.
      \end{enumerate}
  \end{enumerate} 
 \end{lemma}

 We prove  the contrapositive of the second statement of the proposition. Suppose that $\hat \Gamma(G,\calp, S)$ is not fine at cone vertices. Then there is $P\in\calp$ such that $(T_P\hat\Gamma, \angle_P )$ is  not locally-finite.  Let $r>0$ and let  $\{g_i\} \subseteq P$ be an infinite subset such that $\angle_P (g_i,g_j) \leq r$ for every $i,j$. Let $\alpha_{i,j}$ be a path in $\hat\Gamma(G,\calp, S)$ from $g_i$ to $g_j$ of length at most $r$ that does not contain the cone vertex $P$. Let $Q$ denote the left coset $\dot{q}(P)$.  Let $\gamma_i$ be a geodesic in $\Gamma (H, T_0)$ from an element $h_i$ of $Q$ to $q(g_i)$ such that   $\dist_H(h_i, q(g_i)) =\dist_H(Q, q(g_i) )$. Since $q$ is a $(L,C,M)$ quasi-isometry of pairs, each $\gamma_i$ has length at most $M$.
 
Let us prove that the set $\{h_i\}$ is infinite. 
Suppose, for contradiction,  that $\{h_i\}$ is a finite set. Since $T_0$ is a finite generating set, $\Gamma(H, T_0)$ is a locally finite graph and hence it admits only finitely many paths of length at most $M$ with initial vertex in $\{h_i\}$. Since each $\gamma_i$ has length at most $M$ and initial vertex in $\{h_i\}$, it follows that the set $\{q(g_i)\}$ is finite and in particular, bounded. Since $q$ is a quasi-isometry $\Gamma(G,S_0)\to \Gamma(G,T_0)$, it follows that the set $\{g_i\}$ is a  bounded subset of vertices in the locally finite graph  $\Gamma(G,S_0)$, hence the set $\{g_i\}$ is finite, a contradiction.  

To conclude the proof, we show that $\hat \Gamma (H,\calq, T)$ is not fine at the cone vertex $Q$. Since $\{h_i\}$ is an infinite subset of $Q$, it is enough to show that  $\angle_Q(h_i, h_j)\leq r \hat L + M$ for any $i,j$. Consider the path $\beta_{i,j}$ from $h_i$ to $h_j$ obtained as the concatenation of the path $\gamma_i$ from $h_i$ to $\hat q(g_i)$, followed by the path $\hat q(\alpha_{i,j})$ from $\hat q(g_i)$ to $\hat q(g_j)$, and then the path $\bar \gamma_j$ from $\hat q(g_j)$ to $h_j$. The paths $\gamma_i$ and $\gamma_j$ have length bounded by $M$, and they do not contain the cone vertex $Q$ as they are paths in $\Gamma(H,T_0)$; the path $\hat q(\alpha_{i,j})$ has length at most $r\hat L$ and does not contain the cone vertex $Q$ by Lemma~\ref{lem:QI-HatGamma-2}.  Therefore $\angle_Q(h_i,h_j) \leq |\gamma_i|+|\hat q(\alpha_{i,j})|+|\gamma_j| \leq 2M+r\hat L$ as desired.
\end{proof}

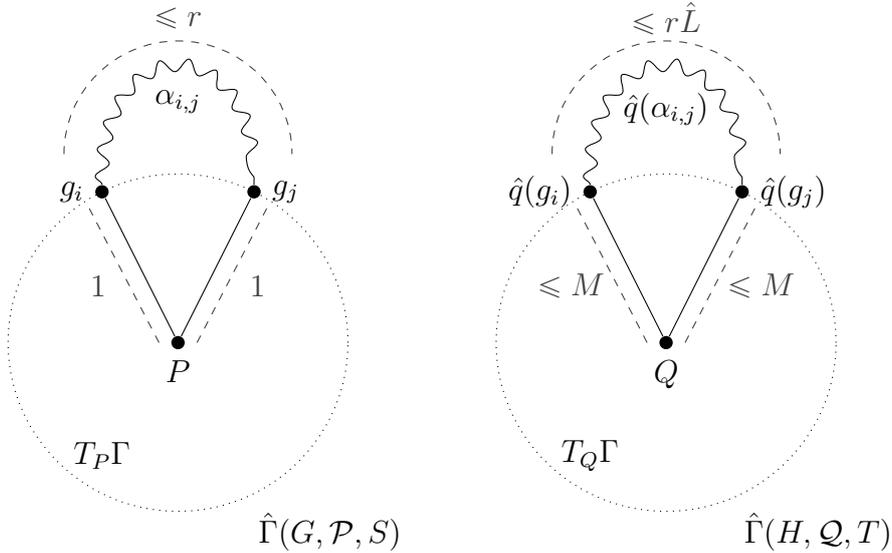
\begin{figure}[h]
    \centering
    \begin{tikzpicture}
    \node[circle,fill=black,inner sep=0pt,minimum size=5pt,label=below:{$P$}] at (0,0) (P) {};
    \draw[color=black, dotted] (P) circle (2.2361);
    \node at (-1,-1.5) (TPG) {$T_P\Gamma$};
    \node at (2,-2.5) (GGPS) {$\hat\Gamma(G,\calp,S)$};
    
    \node[circle,fill=black,inner sep=0pt,minimum size=5pt,label=left:{$g_i$}] at (-1,2) (gi) {};
    \node[circle,fill=black,inner sep=0pt,minimum size=5pt,label=right:{$g_j$}] at (1,2) (gj) {};
    \draw (P) --  (gi);
    \draw (P) -- (gj);
    \draw[draw=black, snake it] (gi) to[out=90,in=90, distance=2cm] (gj);
    \node at (0,3.2) (aij) {$\alpha_{i,j}$};
    
    \draw[color=black!75,dashed] (-1.5,2.5) arc (180:0:1.5);
    \node[color = black!75] at (0,4.3) (r) {$\leq r$};
    
    \draw[color=black!75, dashed] (P) + (-0.25,0) -- ($(P)+ (-0.25,0)!1.9cm!(gi) + (-0.25,0)$);
    \draw[color=black!75, dashed] (P) + (0.25,0) -- ($(P)+ (0.25,0)!1.9cm!(gj) + (0.25,0)$);
    \node[color=black!75] at (-1.05,0.75) (M1) {$1$};
    \node[color=black!75] at (1.05,0.75) (M2) {$1$};
    
    \end{tikzpicture} $\quad \quad$
    \begin{tikzpicture}
    \node[circle,fill=black,inner sep=0pt,minimum size=5pt,label=below:{$Q$}] at (0,0) (P) {};
    \draw[dotted] (P) circle (2.2361);
    \node at (-1,-1.5) (TPG) {$T_Q\Gamma$};
    \node at (2,-2.5) (GGPS) {$\hat\Gamma(H,\calq,T)$};
    
    \node[circle,fill=black,inner sep=0pt,minimum size=5pt,label=left:{$\hat q(g_i)$}] at (-1,2) (gi) {};
    \node[circle,fill=black,inner sep=0pt,minimum size=5pt,label=right:{$\hat q(g_j)$}] at (1,2) (gj) {};
    \draw (P) --  (gi);
    \draw (P) -- (gj);
    \draw[draw=black, snake it] (gi) to[out=90,in=90, distance=2cm] (gj);
    \node at (0,3.1) (aij) {$\hat q(\alpha_{i,j})$};
    
    \draw[color=black!75, dashed] (-1.5,2.5) arc (180:0:1.5);
    \node[color = black!75] at (0,4.3) (r) {$\leq r\hat L$};
    
    \draw[color=black!75, dashed] (P) + (-0.25,0) -- ($(P)+ (-0.25,0)!1.9cm!(gi) + (-0.25,0)$);
    \draw[color=black!75, dashed] (P) + (0.25,0) -- ($(P)+ (0.25,0)!1.9cm!(gj) + (0.25,0)$);
    \node[color=black!75] at (-1.25,0.75) (M1) {$\leq M$};
    \node[color=black!75] at (1.25,0.75) (M2) {$\leq M$};
    
    \end{tikzpicture}
    \caption{Illustration of the proof of Proposition~\ref{prop:FineQI-2}}
    \label{fig:my_label}
\end{figure}

The goal for the remainder of this section is to give algebraic conditions on $\calp$ and $\calq$ to ensure $\dot q$ is a bijection.  The following key definition will provide such a criteria.

\begin{definition}[Reduced Collection]
A collection of subgroups $\calp$ of a group $G$ is \emph{reduced} if  
for any $P,Q\in \calp$ and $g\in G$, then $P$ and $gQg^{-1}$ being commensurable subgroups  implies $P=Q$ and $g\in P$. 
\end{definition}

\begin{remark}
If $\calp$ is a reduced collection of subgroups of a group $G$, then $P=\Comm_G(P)$ for any $P\in\calp$.
\end{remark}

 \begin{proposition}\label{prop:DotqIsFunction}
 Let $q\colon (G, \mathcal{P})\to (H, \mathcal{Q})$ be a $(L,C,M)$-quasi-isometry of  pairs. Then 
 \begin{enumerate}
     \item $\dot{q}$ is a surjective function $G/\calp \to H/\calq$ if $\calq$ is reduced. 
     \item $\dot{q}$ is a bijection $G/\calp \to H/\calq$ if $\calp$ and $\calq$ are reduced. 
 \end{enumerate}
 \end{proposition}
 There are different  versions of the following lemma in the literature: \cite[Lemma 2.2]{MSW11},  
\cite[Lemma 4.7]{MP07} and  \cite[Proposition 9.4]{HK10}, the statement below is taken from the later reference.  For $A\subset G$, $\mathcal{N}_k(A)$ denotes the closed neighborhood of $A$ in $(G,\dist_G)$. 
 \begin{lemma}\label{lem:Perpendicular}
Let $G$ be a finitely generated group with word metric $\dist_G$. Let $gP$ and $fQ$ are arbitrary left cosets of subgroups of $G$. Then for any $k>0$ there is $M>0$ such that 
\[ \mathcal{N}_k(gP)\cap \mathcal{N}_k(fQ)  \subseteq \mathcal{N}_M(gPg^{-1}\cap fQf^{-1}).\]
\end{lemma}

 \begin{lemma}\label{lem:CommensurableSubgps}
 Let $G$ be a finitely generated group with a word metric $\dist_G$, let  $P$ and $Q$ be subgroups, and let $g\in G$. 
 Then $P$ and $gQg^{-1}$ are commensurable subgroups if and only if  $\Hdist_G(P,gQ)<\infty$.
 \end{lemma}
 \begin{proof}
 Suppose $K$ is a finite index subgroup of $P$ and $gQg^{-1}$. Then $\Hdist(K,P)<\infty$ and $\Hdist(K, gQg^{-1})$ are finite. Since $\Hdist(gQg^{-1},gQ) \leq \dist(1,g)<\infty$, it follows that
 \[\Hdist(P,gQ)\leq \Hdist(P,K)+\Hdist(K,gQg^{-1})+\Hdist(gQg^{-1},gQ)<\infty.\]
 
Conversely, suppose $\Hdist(P,gQ)$ is finite. Then  $P \subset  P\cap \caln_k(gQ)$ for some $k$, and therefore Lemma~\ref{lem:Perpendicular} implies that $P \subseteq \caln_M(P\cap gQg^{-1})$ for some $M$. It follows that $P\cap gQg^{-1}$ is a finite index subgroup of $P$. In an analogous way one shows that $P\cap gQg^{-1}$ is a finite index subgroup of $gQg^{-1}$. Whence, $P$ and $gQg^{-1}$ are commensurable subgroups.
 \end{proof}

 \begin{proof}[Proof of Proposition~\ref{prop:DotqIsFunction}]
 To prove the first statement, we only need to show that the relation $\dot{q}$ is a function. Suppose that $\calq$ is reduced and the pairs $(A, h_1Q_1)$ and $(A,h_2Q_2)$ belong to $\dot{q}$. Then   $h_1Q_1,h_2Q_2\in H/\calq$ and $\Hdist_H (h_1Q_1,h_2Q_2)<\infty$.  Lemma~\ref{lem:CommensurableSubgps} implies that $h_1Q_1h_1^{-1}$ and $h_2Q_2h_2^{-1}$ are commensurable subgroups. Since $\calq$ is reduced, it follows that $Q_1=Q_2$ and $h_2\in h_1Q_1$. In particular, $h_1Q_1=h_2Q_2$ and hence $\dot{q}$ is a function. The second statement of the lemma follows from the first one. 
 \end{proof}


We are now ready to prove \Cref{thmx:fine} from the introduction.

\begin{thm}[\Cref{thmx:fine}]\label{thm:fine}
Let $q\colon (G, \calp ) \to (H, \calq)$ be a quasi-isometry of pairs.
Suppose  $\calp$ and $\calq$ are reduced finite collections. Then there is an induced  quasi-isometry of graphs $\hat q\colon \hat \Gamma (G,\calp) \to \hat \Gamma (H,\calq)$, and if $\hat \Gamma (H,\calq)$ is a fine graph then $\hat \Gamma (G,\calp)$ is a fine graph.
\end{thm}
\begin{proof}
The result follows from applying \Cref{prop:DotqIsFunction} to \Cref{cor:FineQI}.
\end{proof}

\section{Almost malnormal collections and quasi-isometries of pairs}
In this section we will prove \Cref{thmx:malnormal} from the introduction.  First, we introduce a refinement $\calp^\ast$ of a collection $\calp$.  In \Cref{lem:QI-Refinement} we show under mild hypothesis $(G,\calp)$ and $(G,\calp^\ast)$ are quasi-isometric pairs under the identity map.

\begin{definition}
 Let $\calp$ be a collection of subgroups of group $G$.  A \emph{refinement} $\calp^\ast$ of $\calp$ is a set of representatives of conjugacy classes of the collection of subgroups \[\{\Comm_G(gPg^{-1}) \colon P\in\calp \text{ and } g\in G \}.\] 
 \end{definition} 
 
 \begin{remark}
 Observe that for a collection of subgroups $\calp$ of a group $G$, there is a refinement $\calp^\ast$ such that each of its elements are of the form $\Comm_G(P)$ for some $P\in \calp$. This is a consequence of  $\Comm_G(gPg^{-1})=g\Comm_G(P)g^{-1}$ for each subgroup $P$ of $G$.
 \end{remark}

\begin{proposition}\label{lem:QI-Refinement}
Let $\calp^*$ be a refinement of a finite collection of subgroups $\calp$ of a finitely generated group $G$. If $P$ is a finite index subgroup of $\Comm_G(P)$ for every $P\in\calp$, then $(G,\calp)$ and $(G,\calp^*)$ are quasi-isometric pairs via the identity map on $G$. 
\end{proposition}
\begin{proof}
Let $\calp=\{P_1,\ldots,P_k\}$. By the previous remark we may assume that every subgroup in $\calp^\ast$ is of the form $\Comm_G(P)$ for some $P\in\calp$. Let $q\colon G\to G$ be the identity map. Since $q$ is a $(1,0)$-quasi-isometry, it is enough to show that there is $M>0$ such that the relation 
\[\dot{q}=\{ (A,B)\in G/\calp \times G/\calp^* \colon \Hdist(A,B)<M \}\]
satisfies that it projects surjectively on $G/\calp$ and on $G/\calp^*$.  

For any $P_i\in \calp$, note that $\Hdist(P_i, \Comm_G(P_i))<\infty$ since $P_i$ has finite index in $\Comm_G(P_i)$. Let  
\[M_1=\max\{\Hdist(P_i,\Comm_G(P_i)) \colon 1\leq i\leq k\}.\]
By definition of $\calp^*$, for any $P_i$, there is $Q_i\in \calp^*$ and $g_i\in G$ such that $\Comm_G(P_i)=g_iQ_ig_i^{-1}$. In particular $\Hdist(\Comm_G(P_i), g_iQ_i)$ is finite. Let
\[ M_2 = \max\{\Hdist(\Comm_G(P_i), g_iQ_i) \colon 1\leq i\leq k\}.\]

Let $M>M_1+M_2$.
Then for any $gP_i\in G/\calp$, $(gP_i, gg_iQ_i)\in \dot{q}$. On the other hand, if $gQ\in G/\calp^*$ then $Q=\Comm_G(P)$ for some $P\in \calp$ and hence $(gP,gQ)\in \dot{q}$. 
\end{proof}

\begin{remark}
Note that in the previous proposition if $\calp$ is infinite the map $\dot q:G/\calp\to G/\calp^\ast$ must  be finite-to-one.  Otherwise after conjugating, there will be a sequence of subgroups $P_i\leq \Comm_G(P_0)$ such that $|\Comm_G(P_0):P_i|\to\infty$, in particular, the sequence of Hausdorff distances $\Hdist(\Comm_G(P_0),P_i)$ is not bounded.
\end{remark}

\begin{definition}\label{def:malnormal}
A collection of subgroups $\calp$ of a group $G$ is \emph{almost malnormal} if for any $P,P'\in\calp$ and $g\in G$, either $gPg^{-1} \cap P'$ is finite, or $P=P'$ and $g\in P$. 
\end{definition}

\begin{remark}
If $\calp$ is an almost malnormal collection of infinite subgroups of a group $G$, then $\calp$ is reduced. 
\end{remark} 

\begin{remark}
If a group $G$ acts by automorphisms on a fine graph $\Gamma$ such that edge stabilizers are finite and $\calp$ is a collection of representatives of conjugacy classes of vertex stabilizers, then $\calp$ is an almost malnormal collection.
\end{remark}

\begin{proposition}\label{prop:MalnormalityQIinvariance}
 Let $q\colon (G, \mathcal{P})\to (H, \mathcal{Q})$ be a quasi-isometry of pairs.  If $\calq$ is an almost malnormal finite collection of infinite  subgroups and $\calp$ is a finite collection, then any refinement $\calp^*$ of $\calp$ is almost malnormal.
\end{proposition}

The proof of Proposition~\ref{prop:MalnormalityQIinvariance} relies on the following lemmas. 

\begin{lemma}\label{lem:RefinementImpliesReduce}
Let $\calp$ be a collection of subgroups of a group $G$.
Suppose $P$ is a finite index subgroup of $\Comm_G(P)$ for every $P\in\calp$. Then any  refinement $\calp^*$ of $\calp$ is a reduced collection.
\end{lemma}
\begin{proof}
Since commensurable subgroups have equal commensurator,
\[\Comm_G(\Comm_G(P))  =\Comm_G(P)\] for every $P\in\calp$.  
Let $P_1, P_2 \in \calp$ such that $\Comm_G(P_1)$ and $\Comm_G(P_2)$ are in $\calp^*$, and let  $g\in G$. Suppose $\Comm_G(P_1)$ and $g\Comm_G(P_2)g^{-1}$ are commensurable subgroups. Then  
\begin{equation}\nonumber
    \begin{split}
   \Comm_G(P_1)  &  = \Comm_G (\Comm_G(P_1))  \\
  &  =\Comm_G(g\Comm_G(P_2)g^{-1}) \\
  & =\Comm_G(\Comm_G(gP_2g^{-1})) \\
  & = \Comm_G(gP_2g^{-1})\\      &= g\Comm_G(P_2)g^{-1}.
    \end{split}
\end{equation}
Since, by definition,  $\calp^*$ does not have two subgroups that are conjugate to each other, it follows that $\Comm_G(P_1) =\Comm_G(P_2)$ and $g\in\Comm_G(P_1)$. Hence $\calp^*$ is reduced. 
\end{proof}

\begin{lemma}\label{lem:GeomMalnormal}
Let $\calp$ be a finite collection of infinite subgroups of a finitely generated group $G$. Then $\calp$ is  almost malnormal  if and only if for any $A,B\in G/\calp$, either $A=B$ or $\caln_n(A)\cap\caln_n(B)$ is a finite subset of $G$ for every $n$.
\end{lemma}
\begin{proof}
Suppose that $\calp$ is an almost malnormal collection of infinite subgroups. Let $g_1P_1, g_2P_2\in G/\calp$   and suppose that $\caln_n(g_1P_1)\cap\caln_n(g_2P_2)$ is an infinite (and hence unbounded) subset of $G$ for some integer $n$. By Lemma~\ref{lem:Perpendicular}, there is an integer $m$ such that $\caln_n(g_1P_1)\cap\caln_n(g_2P_2) \subset \caln_m(g_1P_1g_1^{-1}\cap g_2P_2g_2^{-1})$. It follows that $g_1P_1g_1^{-1}\cap g_2P_2g_2^{-1}$ is an infinite subgroup and hence $P_1=P_2$ and $g_1^{-1}g_2\in P_1$ by almost malnormality. Therefore $g_1P_1=g_2P_2$.  

Conversely, suppose that for any $A,B\in G/\calp$, either $A=B$ or $\caln_n(A)\cap\caln_n(B)$ is a finite set for every $n$. Let   $P, P'\in \calp$ and $g\in G$ and suppose that $gPg^{-1} \cap P'$ is an infinite subgroup.
It follows that there is $n>0$ such that $\caln_n(gP) \cap \caln_n (P')$ is an infinite subset of $G$. Hence $gP=P'$ and in particular $P=P'$ and $g\in P$.
\end{proof}

\begin{lemma}\label{lem:FiniteIndex}
Let $q\colon (G, \mathcal{P})\to (H, \mathcal{Q})$ be a quasi-isometry of pairs. Suppose that $\calp$ and $\calq$ are finite collections, and $\calq$ is reduced.  If $Q$ is finite index in $\Comm_H(Q)$ for every $Q\in \calq$, then $P$ is finite index in $\Comm_G(P)$ for every $P\in\calp$.
\end{lemma}
\begin{proof}
Since $\calq$ is reduced, $\dot{q}$ is a function from $G/\calp \to G/\calq$. Since both $\calp$ and $\calq$ are finite collections, it follows that $\dot{q}\colon (G/\calp,\Hdist) \to (H/\calq,\Hdist)$ is a quasi-isometry between locally finite metric spaces. 
Suppose that $P\in \calp$ has infinite index in $\Comm_G(P)$.  Lemma~\ref{lem:CommensurableSubgps} implies that there is an infinite collection of left cosets $\cala=\{g_iP \colon i\in I\}$ such that $\Hdist(g_iP,g_jP)<\infty$ for any $i,j\in I$. By local finiteness of $(G/\calp, \Hdist)$, the collection $\cala$ is an unbounded subset of $G/\calp$. It follows that $\calb=\{\dot{q}(g_iP)\colon i\in I\}$ is an unbounded subset of $H/\calq$. Since $\calq$ is a  finite collection, and $\dot{q}(g_iP)=h_iQ_i$ for some $h_i\in H$ and $Q_i\in\calq$, the pigeon hole principle implies that we can assume that all $Q_i$'s are a fixed $Q\in \calq$. By Lemma~\ref{lem:CommensurableSubgps}, the subgroup $Q$ has infinite index in $\Comm_H(Q)$. 
\end{proof}

\begin{proof}[Proof of  Proposition~\ref{prop:MalnormalityQIinvariance}]
Suppose that $q\colon (G,\calp)\to (G,\calq)$ is a quasi-isometry of pairs. 
Since $\calq$ is an almost malnormal collection of infinite subgroups, it is a reduced collection and every element of $\calq$ has finite index in its commensurator. Since $\calp$ and $\calq$ are finite collections,    Lemma~\ref{lem:FiniteIndex} implies that every element of $\calp$ has finite index in its commensurator.
Let $\calp^*$ be a refinement of $\calp$ in $G$.
By Proposition~\ref{lem:QI-Refinement} there is a quasi-isometry of pairs  $p\colon (G,\calp^*) \to (G,\calp)$. Then the composition $r=p\circ q$ is an $(L,C,M)$-quasi-isometry of pairs $(G,\calp^*) \to (H,\calq)$.
Lemma~\ref{lem:RefinementImpliesReduce} implies that $\calp^*$ is a reduced collection.
Therefore $\dot{r}$ is a bijection $G/\calp^* \to H/\calq$ by Proposition~\ref{prop:DotqIsFunction}.  To conclude that $\calp^*$ is an almost malnormal we verify the hypothesis of Lemma~\ref{lem:GeomMalnormal}.

\textbf{Claim:} \emph{$\calp^*$ is a finite collection of infinite subgroups.}

Since $\calp$ is finite, then $\calp^*$ is finite. Every element of $\calp^*$ is a conjugate of a subgroup of the form $\Comm_G(P)$ for some $P\in \calp$, hence it is enough to show that $\calp$ contains only infinite subgroups. Observe that any $P\in \calp$ is an infinite subgroup since $\Hdist(\dot{q}(P), Q)<\infty$ for some $Q\in H/\calq$ and every subgroup in $\calq$ is infinite. $\blackdiamond$

\textbf{Claim:} \emph{For any $A,B\in G/\calp^*$, either $A=B$ or $\caln_n(A)\cap\caln_n(B)$ is a finite subset of $G$ for every $n$.}

Let $A,B\in G/\calp^*$ and suppose that $A\neq B$. Since $\dot{r}\colon G/\calp^*\to H/\calq$ is a bijection, it follows that $\dot{r}(A)$ and $\dot{r}(B))$  are distinct elements of $H/\calq$. Since $\calq$ is an almost malnormal collection, Lemma~\ref{lem:GeomMalnormal} implies that for any  integer $m$ the intersection 
$\caln_m(\dot{r}(A))\cap\caln_m(\dot{r}(B))$ is a finite (and hence bounded) subset of $H$.  Since $r\colon G\to H$ is a quasi-isometry, it follows that for every $n$, the intersection $\caln_n(A)\cap\caln_n(B)$ is a bounded (and hence finite) subset of $G$. $\blackdiamond$
\end{proof}

\begin{thm}[\Cref{thmx:malnormal}]\label{thm:malnormal}
 Let $q\colon (G, \mathcal{P})\to (H, \mathcal{Q})$ be a quasi-isometry of pairs.  If $\calq$ is an almost malnormal finite  collection of infinite  subgroups and $\calp$ is a finite collection, then any refinement $\calp^*$ of $\calp$ is almost malnormal and $q\colon (G,\calp^*) \to (H,\calq)$ is a quasi-isometry of pairs.
\end{thm}
\begin{proof}
The result follows from \Cref{lem:QI-Refinement} and \Cref{prop:MalnormalityQIinvariance}.
\end{proof}

\section{Examples and non-examples}
In this section we show that there are examples of pairs $(G,H)$ with well-defined relative Dehn function outside of the context of relatively hyperbolic groups. Hyperbolically embedded subgroups were introduced in~\cite{DGOsin2017} by Dahmani, Guirardel and Osin.  Given a group $G$, $X\subset G$ and $H\leq G$, let $H\hookrightarrow_h (G,X)$ denote that $H$ is a hyperbolically embedded subgroup of $G$ with respect to $X$. There is a characterisation in~\cite{MPR2021} of $H$ being hyperbolically embedded into $G$
that fits into the context of our \Cref{thmx:ConefOff}, namely, in terms of fine vertices in coned-off Cayley graphs (see \Cref{def:fine2}).

\begin{proposition}  \emph{\cite[Proposition 1.4]{MPR2021}}  \label{prop:q2}
Let $G$ be a group, $X\subset G$ and $H\leq G$. Then $H \hookrightarrow_h (G,X) $ if and only if   $\hat{\Gamma}(G,H,X)$ is connected, hyperbolic, and fine at cone vertices. 
\end{proposition}

The following theorem provides our examples.

\begin{thm}[\Cref{thmx:last}]\label{thm:last}
Let $G$ be a finitely presented group and $H\leq G$ be a subgroup. If $H\hookrightarrow_h G$ then the relative Dehn function $\Delta_{G,H}$ is well-defined. 
\end{thm}

The proof of the theorem is discussed after the following lemma. 

\begin{lemma}\label{lem:finpresiff}
Let $G$ be a finitely generated group and $H$ a finitely presented subgroup. Then $G$ is finitely presented if and only if $G$ is finitely presented relative to $H$. 
\end{lemma}
\begin{proof}
Suppose that $G$ has a finite presentation $\langle A\ |\ R \rangle$.  
Let $R_H$ be the collection of all relations in $H$ over the generating set $H-\{1\}$, that is, $H=\langle H-\{1\}\ |\ R_H\rangle$. Let $\{h_1,\ldots,h_k\}$ be a finite generating set of $H$. Then, there is a word $w_i$ over the alphabet $A$ that represents $h_i$. Observe that \[\langle A \sqcup (H-\{1\})\ |\  R, R_H, h_1=w_1,\ldots , h_k=w_k \rangle\]
yields a finite relative presentation of $G$ with respect to $H$.

Conversely, suppose that $\langle A, H\ | \  R  \rangle$ is a finite relative presentation of $G$ with respect to $H$, and let $\langle B \ |\ T \rangle$ be a finite presentation of $H$. Then 
$\langle A\sqcup B, H\ |\ R\sqcup T, h_1=w_1,\ldots , h_k=w_k \rangle$ is a finite relative presentation of $G$ with respect to $H$, where  $\{h_1,\ldots , h_k\} \subset H$ is a finite generating set of $H$ and $w_i$ is a word over $B$ that represents the element $h_i$ (after choosing an isomorphism $F(B)/\nclose{T} \to H$). This relative presentation yields a standard presentation
$\langle A\sqcup B\sqcup (H-\{1\})\ |\ R\sqcup T \sqcup R_H \sqcup \{  h_1=w_1,\ldots , h_k=w_k \}  \rangle$
of $G$, where $R_H$ is the collection of all relations in $H$ over the generating set $H-\{1\}$. Since the $\{h_1,\ldots, h_k\}$ generate $H$, using Tietze transformations one obtains that  
$\langle A\sqcup B \ |\ R\sqcup T \sqcup \{  h_1=w_1,\ldots , h_k=w_k \}  \rangle$ is a presentation of $G$ which is finite.
\end{proof}

\begin{proof}[Proof of Theorem~\ref{thm:last}]
First, note that the theorem is trivial in the case that $H$ is a finite subgroup of $G$.  Indeed, any finite subgroup is hyperbolically embedded by definition and a finite relative presentation of a group with respect to a finite subgroup is in fact a finite presentation.  In particular, the relative Dehn function coincides with the Dehn function and the Dehn function of a finitely presented group is always well-defined.

Since $G$ is finitely presented and $H\hookrightarrow_h G$, it follows from \cite[Corollary~4.32]{DGOsin2017} that $H$ is finitely presented. Hence, by \Cref{lem:finpresiff}, $G$ is finitely presented relative to $H$.  

Let $S$ be a finite generating set of $G$. In view of~\Cref{thmx:ConefOff}\eqref{thmx:ConefOff:2}, to conclude that $\Delta_{G,H}$ is well-defined, it is enough to prove that $\hat\Gamma(G,\calp,S)$ is a fine graph. 

Suppose that $H\hookrightarrow_h (G, X)$ for some $X\subset G$. Without loss of generality, assume that $X$ contains the finite generating set $S$, see~\cite[Corollary 4.27]{DGOsin2017}. It follows that $\hat\Gamma(G,H,S)$ is a subgraph of $\hat\Gamma(G,H,X)$. Since $S$ is finite, observe that every vertex of $\hat\Gamma(G,H,S)$ has either  finite degree or is cone-vertex. By \Cref{prop:q2}, the graph $\hat\Gamma(G,H,X)$ is fine at every cone vertex, and hence so is $\hat\Gamma(G,H,S)$. Therefore $\hat\Gamma(G,H,S)$ is a fine graph.
\end{proof}

\begin{example}
In \cite{G19} the author shows that amongst graph products of finite groups various \emph{eccentric subgroups} (see loc. cit. for a definition) are quasi-isometrically rigid in the sense of \cite{MaSa21}.  Let $G$ be a graph product of finite groups that is not virtually cyclic or a direct product of two infinite groups, then $G$ is acylindrically hyperbolic.  Suppose $H$ is an eccentric subgroup, then $H\hookrightarrow_h G$ if and only if $H$ is almost malnormal.  In particular, if $H$ is almost malnormal, then by \Cref{thm:last}, we see that $\Delta_{G,H}$ is well-defined.  Moreover, for any graph product of finite groups $G'$ quasi-isometric to $G$, there exists a subgroup $H'<G'$, such that $\Delta_{G',H'}\asymp\Delta_{G,H}$.
\end{example}

The following example demonstrates that $\Delta_{G,\calp}$ being well-defined is not implied by $\calp$ being a qi-characteristic collection in the sense of \cite{MaSa21}. 

\begin{example}\label{ex:wreath}
Let $F$ be a finite group and let $H$ be a finitely presented one-ended group. Consider the wreath product $G=F\wr H$.
In work of Genevois and Tessera~\cite[Proof of Theorem~7.1]{GT21}, they show that an quasi-isometry of $q\colon G \to G$ is a quasi-isometry of pairs $q\colon (G, H) \to (G,H)$. Moreover, $H$ is an almost malnormal subgroup and in fact is qi-characteristic in the sense of~\cite{MaSa21}, see~\cite[Theorem~1.18]{GT21}. However the coned-off Cayley graph of $G$ with respect to $H$ is not fine, so the group $G$ can not have a well-defined Dehn function by \Cref{thm:ConefOff}.  To prove this, suppose $F$ is the group with two elements and let $H$ be a group with an element of infinite order $a$. Consider the 
wreath product $G=F\wr H$. If $F$ has non-trivial element $x$, then $G$ has a relative presentation 
\[\langle x, H\ | \ x^2,\  [x,gxg^{-1}] \text{ for all $G-\{e\}$}  \rangle .\]
Let us observe that the coned-off Cayley graph $\hat\Gamma(G,H,\{x\})$ is not fine. Consider the edge $\{e,H\}$. We will show that there infinitely many circuits of length twelve that contain this edge, each of them induced by a word
\[ w_n= xa^nxa^{-n}xa^nxa^{-n}\]
which represents the identity.
For an arbitrary integer $n>0$, the sequence of vertices
\[ \gamma_n = [ H, e, x, xH, xa^n, xa^nx, xa^{n}xH, xa^{n}xa^{-n},  a^nxa^{-n}, a^nxH, a^nx, a^n, H] \]
is a closed path of length twelve in $\hat\Gamma$ containing the edge $\{e,H\}$; the only non-trivial adjacency follows from  $xa^{n}xa^{-n}=a^nxa^{-n}$. It follows that $\gamma_n$ is a circuit since one can show that  the left cosets $H, xH, xa^nxH, a^nxH$ are all distinct. On the other hand, $a^nxH = a^mx H$ if and only if $n=m$, and therefore $\gamma_n\neq \gamma_m$ if $m\neq n$.  Note, we do not know the existence of a finite relative presentation for $G$ with respect to $H$, but observe that we do not use this in the remark.
\end{example}

Finally, we will show the relative Dehn function of $\mathrm{BS}(k,l)$ with respect to the stable letter is not well-defined if either $k$ or $l$ divides the other one.

\begin{example}\label{prop.BS.nwd}
Let $G=\mathrm{BS}(k,l)=\langle a,t\ | ta^kt^{-1}=a^l\rangle$.  We claim that if $k\mid l$ or $l\mid k$, then $\Delta_{G,\langle t\rangle}$ is not well-defined.  As in the previous example we will show that the coned-off Cayley graph $\hat\Gamma(G,\langle t\rangle,\{a,t\})$ is not fine and apply \Cref{thm:ConefOff}.

Without loss of generality let $\ell=km$ and consider $w_n=t^na^kt^{-n}at^na^{-k}t^{-n}a^{-1}$.  Observe that $w_n=1_G$ since $t^na^kt^{-n}=a^{k\ell^n}$ and $t^na^{-k}t^{-n}=a^{-k\ell^n}$.  The word $w_n$ describes a circuit of length $2k+6$ in $\hat\Gamma(G,\langle t\rangle,\{a,t\})$ because the four left cosets $\langle t\rangle$, $t^na^k\langle t\rangle =a^{k\ell^n}\langle t\rangle$, $t^na^kt^{-n}a\langle t\rangle=a^{k\ell^n+1}\langle t\rangle$, and $t^na^kt^{-n}a^{-k}\langle t\rangle =a \langle t\rangle$ are all distinct.  In particular, the coned-off Cayley graph $\hat\Gamma(G,\langle t\rangle,\{a,t\})$ is not fine.
\end{example}

\appendix
\section{Relative Dehn functions of Baumslag-Solitar groups}
\smallskip
\begin{center}by \textsc{Ashot Minasyan}\end{center}
\medskip
For two non-zero integers  $k,l$ we define the Baumslag-Solitar group $BS(k,l)$ by the presentation
\begin{equation*}
BS(k,l)=\langle a,t \mid t a^k t^{-1}=a^l\rangle.
\end{equation*}

Evidently $BS(k,l)$ is finitely presented relative to its cyclic subgroup $\langle t \rangle$ and we can consider the relative presentation
\begin{equation}\label{eq:BS-rel_pres}
BS(k,l)=\langle a,\langle t \rangle \mid \mathcal{R}\rangle,
\end{equation}
where $\mathcal{R}$ consists of all cyclic permutations of the relator $t a^k t^{-1}a^{-l}$ and its inverse.

Let $F=F(a,t)$ be the free group freely generated by $\{a,t\}$. The generating set $\{a\} \cup \langle t \rangle$ of $F$ gives rise to the
relative word length $\|\cdot\|_{\{a\} \cup \langle t \rangle}$ for words over the alphabet $\{a\}^{\pm 1} \cup \langle t \rangle$.

The goal of this appendix is to provide a characterisation for the Dehn function of $BS(k,l)$ with respect to $\langle t \rangle$ to be well-defined (we shall use the definitions of the relative area and relative Dehn functions from Remark~\ref{rem:def_of_rel_D_f}).

\begin{thm} \label{thm:BS_rel_DF-character} Let $G=BS(k,l)$, for some non-zero integers $k,l$.
The relative Dehn function $\Delta_{G,\langle t \rangle}$ is well-defined if and only if $k$ does not divide $l$ and $l$ does not divide $k$.
\end{thm}

\begin{remark} \Cref{thm:BS_rel_DF-character} implies that the relative Dehn function of the group $G=BS(2,3)$ with respect to the cyclic subgroup $\langle t \rangle$ is well-defined. However, we note that $\langle t \rangle \not{\hookrightarrow}_h G$ so the converse of \Cref{thmx:last} is false. In fact, $G$ does not contain any proper infinite hyperbolically embedded subgroups: see Theorem~1.2 and Example~7.4 in \cite{Osin16}.
\end{remark}

\begin{proof}[Proof of Theorem~\ref{thm:BS_rel_DF-character}] The necessity has already been proved in Example~\ref{prop.BS.nwd}, using \Cref{thm:ConefOff}. Below we give a different argument, based on the results of Osin \cite{Osin06}.

Throughout the argument we will use the following well-known elementary facts about $G=BS(k,l)$: the elements $a$ and $t$ have infinite order and $\langle a \rangle \cap \langle t \rangle=\{1\}$ in $G$.

Assume, without loss of generality, that $k$ divides $l$, so that $l=km$, for some $m \in \mathbb{Z} \setminus\{0\}$.
Arguing by contradiction, suppose that the Dehn function $\Delta_{G,\langle t \rangle}$ is well-defined. Then, by \cite[Proposition~2.36]{Osin06}, $\langle t \rangle$ is a malnormal subgroup
of $G$ (i.e., $g \langle t \rangle g^{-1} \cap \langle t \rangle=\{1\}$ for any $g \in G \setminus \langle t \rangle$).

If $l=\pm k$ then $t^2 a^k t^{-2}=a^k$, so that $a^{-k} t^2 a^k =t^2 \in a^{-k} \langle t \rangle a^k \cap \langle t \rangle=\{1\}$, contradicting to the fact that $t$ has infinite order in $G$.
Therefore we can further assume that $|k| \neq |l|$, so that $|m|>1$.

For any $s \in \mathbb{N}$ we have $t^s a^k t^{-s}=a^{m^s k}$ in $G$, whence the commutator word \[W_s=[t^s a^k t^{-s},a]=t^s a^k t^{-s} a t^s a^{-k} t^{-s} a^{-1}\] represents the trivial element of $G$. Note that $\|W_s\|_{\{a\} \cup \langle t \rangle}=2k+6$, so, since the Dehn function $\Delta_{G,\langle t \rangle}$ is well-defined, there exists a constant $C \ge 0$ such that
\[\area(W_s) \le C, ~\text{ for all } s \in \N.\]

For each $s \in \N$ let $q_s$ be the cycle in the Cayley graph $\Gamma(G,\{a\} \cup \langle t \rangle \setminus\{1\} )$ based at the identity element and labelled by the word $W_s$.
By the definition of $W_s$, $q_s$ is a concatenation of eight subpaths $p_1,p_2,\dots,p_8$, where $p_1$ is the edge labelled by $t^k$, $p_2$ has length $|k|$ and is labelled by $a^k$, and so on: see Figure~\ref{fig:q_s}.

\begin{figure}[!ht]
  \begin{center}
   \includegraphics{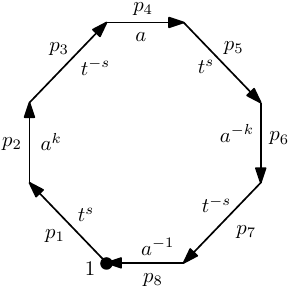}
  \end{center}
  \caption{The cycle $q_s$ (markings inside the polygon represent the labels of the subpaths $p_1,\dots,p_8$)}\label{fig:q_s}
\end{figure}

Using Osin's terminology from \cite[Section 2.2]{Osin06}, we see that $p_1$, $p_3$, $p_5$ and $p_7$ is the list of $\langle t \rangle$-components of $q$. Let us show that $p_1$ is an isolated component of $q_s$. Indeed, if $p_1$ is connected to $p_3$ then the label of $p_2$, $a^k$, must represent an element of $\langle t \rangle$ in $G$. The latter is impossible since $\langle a \rangle \cap \langle t \rangle=\{1\}$ in $G$ and $a^k \neq 1$. Similarly, $p_1$ cannot be connected to $p_7$. Finally, if $p_1$ is connected to $p_5$ then the label of the path $p_1p_2p_3p_4$ must represent an element of $\langle t \rangle$ in $G$. However, this label is equal to $t^s a^k t^{-s}a$, which simplifies to $a^{m^s k+1}$ in $G$. This again yields a contradiction because
$m^s k+1 \neq 0$ (which is true as $|m|>1$ and $s \in \N$).

Therefore we can apply \cite[Lemma 2.27]{Osin06} to the cycle $q_s$, claiming that \[|t^s|_\Omega \le M\area(W_s),\] where $\Omega=\{t,t^{-1}\}$ and $M=\max\{\|R\|_{\{a\} \cup \langle t \rangle} \mid R \in \mathcal{R}\}=|k|+|l|+2$. It follows that $s\le MC$ for all $s \in \N$. This contradiction shows that the Dehn function $\Delta_{G,\langle t \rangle}$ is not well-defined, so the necessity statement of the theorem has been proved.

The proof of the sufficiency occupies the rest of the appendix and will be completed in Theorem~\ref{thm:BS_suff} below.
\end{proof}

\subsection{Notation}\label{subsec:notation}
We will use $\Z$ to denote the set of all integers, $\N=\{1,2,\dots\}$ -- the set of natural numbers and $\N_0=\N \cup \{0\}$.
Given a prime $p$ and an integer $n \in \Z \setminus\{0\}$, we will write \[\nu_p(n)=\max\{s \in \N_0 \mid p^s \text{ divides }n\} \in \N_0.\] Evidently, $\nu_p(n) \le \log_{p}(|n|)$
and $\nu_p(mn) =\nu_p(m)+\nu_p(n)$, for all $m,n \in \Z\setminus\{0\}$.

Further on $k,l$ will be some fixed non-zero integers and $G$ will be the Baumslag-Solitar group $BS(k,l)$, equipped with the relative presentation \eqref{eq:BS-rel_pres}.
For two words $W,W'$ over the alphabet $\{a\}\cup\langle t \rangle$ we will write $W \eqG W'$ if $W$ and $W'$ represent the same element of $G$.

\subsection{Some terminology}\label{subsec:terminology}
By the Normal Form Theorem for free products, we know that any word $W$ over the alphabet $\{a\}^{\pm 1} \cup \langle t \rangle$ is equal in $F=F(a,t)$ to a
unique \emph{freely reduced word}, which has the form
\begin{equation}\label{eq:reduced_W}
a^{u_0} t^{v_1} a^{u_1} \dots t^{v_m} a^{u_m}, \text{ where } m \ge 0,~  u_0,u_m \in \mathbb{Z},~u_1,\dots,u_{m-1},v_1,\dots,v_m \in \mathbb{Z}\setminus\{0\},
\end{equation}
and $t^{v_1},\dots,t^{v_m} \in \langle t \rangle\setminus\{1\}$ are treated as single letters from the alphabet $\{a\}^{\pm 1} \cup \langle t \rangle$. We will call the number $m$ the \emph{syllable length of $W$} and will denote it $\sl(W)$. Observe that  $\sl(W) \leq \|W\|_{\{a\} \cup \langle t \rangle}$ for any freely reduced word $W$.

\begin{defn}[Reduction of the first type]\label{def:red-first}
Suppose that $W$ is a word of the form \eqref{eq:reduced_W}. If for some $i \in \{1,\dots,m-1\}$, we have $v_i>0$, $v_{i+1}<0$ and $u_i \in k\Z\setminus\{0\}$ then we can perform a {reduction of the first type} on $W$ as follows.

Set $s=u_i/k \in \Z$ and observe that, by applying a defining relation from presentation \eqref{eq:BS-rel_pres} $|s|$ times, we get
\[ta^{u_i} t^{-1}\eqG \left(ta^kt^{-1}\right)^s \eqG a^{l s} \text{ in } G. \]
Therefore $W$ equals in $G$ to the word
\begin{equation}\label{eq:W'-first_type}
W'=a^{u_0} t^{v_1} a^{u_1} \dots a^{u_{i-1}} t^{v_i-1} a^{ls} t^{v_{i+1}+1} a^{u_{i+1}}   \dots t^{v_m} a^{u_m}.
\end{equation}
We will say that $W'$ has been obtained from $W$ by applying a \emph{reduction of the first type at place $i$}, writing $W \redone W'$.
\end{defn}

We can similarly define basic reductions of the second type.

\begin{defn}[Reduction of the second type]
Suppose that $W$ is a word of the form \eqref{eq:reduced_W}. If for some $i \in \{1,\dots,m-1\}$, we have $v_i<0$, $v_{i+1}>0$ and $u_i \in l\Z\setminus\{0\}$ then we can perform a {reduction of the second type} on $W$ as follows.

Set $s=u_i/l \in \Z$ and observe that, by applying a defining relation from  presentation \eqref{eq:BS-rel_pres}  $|s|$ times, we get
\[t^{-1}a^{u_i} t \eqG \left(t^{-1} a^l t\right)^s \eqG a^{k s} \text{ in } G. \]
Therefore $W$ equals in $G$ to the word
\begin{equation}\label{eq:W'-second_type}
W'=a^{u_0} t^{v_1} a^{u_1} \dots a^{u_{i-1}} t^{v_i+1} a^{ls} t^{v_{i+1}-1} a^{u_{i+1}}   \dots t^{v_m} a^{u_m}.
\end{equation}
We will say that $W'$ has been obtained from $W$ by applying a \emph{reduction of the second type at place $i$}, writing $W \redtwo W'$.
\end{defn}

When the type of the reduction does not matter we will simply write $W \red W'$.
Note that after applying a reduction (of any type) to a freely reduced word $W$  the resulting word $W'$ satisfies $\sl(W') \leq \sl(W)$. Moreover, if $\sl(W')=\sl(W)$ then the word $W'$
(from \eqref{eq:W'-first_type} or \eqref{eq:W'-second_type})  is again freely reduced in the above sense.

\begin{defn}[Trimming chain] Let $W$ be a freely reduced word over the alphabet $\{a\}^{\pm 1} \cup \langle t \rangle$ and $i \in \{1,\dots,\sl(W)-1\}$. For any $\ell \in \N$, a \emph{trimming chain of the first type at place $i$ of length $\ell$}
is a sequence of reductions
\[W=W_0 \redone W_1\redone \dots \redone W_{\ell-1} \redone W_\ell, \]
where $\sl(W)=\sl(W_1)=\dots=\sl(W_{\ell-1})$ and $\sl(W_\ell)<\sl(W)$.

A \emph{trimming chain of the second type at place $i$ of length $\ell$},
\[W=W_0 \redtwo W_1\redtwo \dots \redtwo W_{\ell-1} \redtwo W_\ell, \]
is defined similarly.
\end{defn}

\subsection{Technical lemmas}\label{subsec:tech_lemmas}
From now on we assume that $k\nmid l$ and $l \nmid k$. In this case we can choose some primes $p,q \in \N$ such that $\nu_p(k)>\nu_p(l)$ and $\nu_q(l)>\nu_q(k)$.

\begin{lemma} \label{lem:trim_chain_length} Let $W$ be the word given by \eqref{eq:reduced_W} with $\sl(W)=m>0$.
If $W$ represents the trivial element of $G$ then there is $i \in \{1,\dots, m-1\}$ such that either
$W$ admits a trimming chain of the first type at place $i$ of length at most $ \nu_p(u_i)$ or it admits a trimming chain of the second type at place $i$ of length at most $\nu_q(u_i)$.
\end{lemma}

\begin{proof}
We will prove the statement by induction on the total number of $t$'s occurring in $W$, i.e., on the number $v(W)=\sum_{r=1}^m |v_r|$.

Since $W \eqG 1$, by Britton's lemma (\cite[Section IV.2]{L-S}), the number $v(W)$ must be at least $2$, and if $v(W)=2$ then $\sl(W)=m=2$ and either $v_1=1$, $v_2=-1$ and
$u_1 \in k\Z \setminus\{0\}$ (i.e., $W$ admits a reduction of the first type) or $v_1=-1$, $v_2=1$ and $u_1 \in l\Z\setminus\{0\}$ (i.e., $W$ admits a reduction of the second type).
Without loss of generality, let us assume that we are in the former case. Applying a reduction of the first type to $W$ we obtain a word $W'$
with $\sl(W')=0<\sl(W)$, so $W \redone W'$ is a trimming chain of the first type at place $1$ of length $1$. Moreover, $1 \le \nu_p(u_1) $, as $p \mid k \mid u_1$, so the base of induction has been established.

Suppose now that $v(W)>2$. By Britton's lemma, $W$ admits a reduction (of some type) at some place $i \in \{1,\dots,\sl(W)-1\}$, and again, without loss of generality, we assume that it is  a reduction of the first type (the other case is similar). Let $W'$ be the word \eqref{eq:W'-first_type} resulting in this reduction.

If $\sl(W')<\sl(W)$ then $W \redone W'$ is a trimming chain of the first type of length $1 \le \nu_p(u_i)$, as required. So we can further assume that $\sl(W')=\sl(W)=m$, whence $W'$ is again freely reduced and $v(W')=v(W)-2<v(W)$. By induction, $W'$ must admit a trimming chain
(of one of the two types) at some place $j \in \{1,\dots,m-1\}$ of length $\ell \in \N$. If $j \neq i$ then we can perform the same trimming chain on $W$ since $u_j$ is not affected by the original reduction $W \redone W'$ and $\sl(W')=\sl(W)$. The desired inequality on $\ell$ will then follow by induction.

Now let us suppose that $j=i$. Since $\sl(W')=\sl(W)$, the trimming chain at  place $i$ for $W'$ must have the same type as the original reduction from $W$ to $W'$, thus we have a trimming chain
\[W'=W_0 \redone W_1\redone \dots \redone W_{\ell-1} \redone W_\ell. \] By precomposing this trimming chain with the original reduction $W \redone W'$, we obtain a trimming chain of the first type at place $i$ of length $\ell+1$ for $W$. By induction and the construction of $W'$ (see \eqref{eq:W'-first_type}), we have $\ell \le \nu_p(ls)$, where $s=u_i/k \in \Z\setminus\{0\}$.
Since $\nu_p(k) \ge \nu_p(l)+1$, we can conclude that
\[\ell+1\le \nu_p(ls)+1=\nu_p(l)+\nu_p(s)+1 \le \nu_p(k)+\nu_p(s)=\nu_p(ks)=\nu_p(u_i).\]

Thus we have established the step of induction, and so the statement is proved.
\end{proof}

Denote
\begin{equation}\label{eq:def_of_alpha}
\alpha=\max\{|l/k|,|k/l|\}>1.
\end{equation}

\begin{lemma}\label{lem:red_seq} Let $W$ be a word of the form \eqref{eq:reduced_W}, representing the trivial element of $G$. Suppose that
\begin{equation}\label{eq:red_seq}
W=W_0 \red W_1\red \dots \red W_{\ell-1} \red W_\ell
\end{equation}
is a sequence of reductions (of the same type) at place $i \in \{1,\dots,\sl(W)-1\}$, where $\ell \in \N$ and $\sl(W)=\sl(W_1)=\dots=\sl(W_{\ell-1})$.
 Denote $n=\|W\|_{\{a\} \cup \langle t \rangle} \in \N$, then
\begin{equation}\label{eq:ineq_for_W_ell}
\|W_\ell\|_{\{a\} \cup \langle t \rangle} \le \alpha^\ell \,n~ \text{ and }~ \area(W) \le \area(W_\ell)+\frac{\alpha^{\ell}-1}{\alpha-1}\,n .
\end{equation}
\end{lemma}

\begin{proof} Without loss of generality we will assume that all of the reductions in the sequence \eqref{eq:red_seq} are of the first type. We will argue by induction on $\ell$.

Suppose, first, that $\ell=1$ and $W_1=a^{u_0} t^{v_1} a^{u_1} \dots a^{u_{i-1}} t^{v_i-1} a^{ls} t^{v_{i+1}+1} a^{u_{i+1}}   \dots t^{v_m} a^{u_m}$, where $s=u_i/k$.
Then
\[n=\|W\|_{\{a\} \cup \langle t \rangle} = m+\sum_{r=0}^m |u_r| ~\text{ and }~ \|W_1\|_{\{a\} \cup \langle t \rangle} \le m+\sum_{r=0,r \neq i}^m |u_r|+|ls|.\]

Since $|ls|/|u_i|=|l/k| \le \alpha$ and $\alpha >1$, we see that $\|W_1\|_{\{a\} \cup \langle t \rangle} \le \alpha\,n$. The word $W_1$ can be obtained from the word $W$ by applying a defining relation from the presentation \eqref{eq:BS-rel_pres} $|s|$ times, so, since $|s| \le |u_i| \le n$, we have
\[\area(W) \le \area(W_1)+|s| \le \area(W_1)+n.\]

Now assume that $\ell \ge 2 $ and set $n_1=\|W_1\|_{\{a\} \cup \langle t \rangle}$. By induction, we know that
\begin{equation}\label{eq:ineq_for_W_ell_from_n_1}
\|W_\ell\|_{\{a\} \cup \langle t \rangle} \le \alpha^{\ell-1} \,n_1~ \text{ and }~ \area(W_1) \le \area(W_\ell)+\frac{\alpha^{\ell-1}-1}{\alpha-1}\,n_1 .
\end{equation}
We have shown above that $n_1 \le \alpha\, n$ and $\area(W) \le \area(W_1)+n$. Combining this with inequalities \eqref{eq:ineq_for_W_ell_from_n_1}, we obtain \eqref{eq:ineq_for_W_ell}.
\end{proof}

\subsection{Proof of the sufficiency in Theorem~\ref{thm:BS_rel_DF-character}}
\begin{thm}\label{thm:BS_suff}
Let $G$ be the Baumslag-Solitar group $BS(k,l)$, for some $k,l \in \mathbb{Z}\setminus\{0\}$. If
neither of $k,l$ divides the other one then the relative Dehn function $\Delta_{G,\langle t \rangle}$ is well-defined.
\end{thm}

\begin{proof} Choose primes $p,q \in \N$ as in the beginning of Subsection~\ref{subsec:tech_lemmas} and let $\alpha>1$ be defined by \eqref{eq:def_of_alpha}.

To prove that $\Delta_{G,\langle t \rangle}$ is well-defined it is sufficient to show that there is a function $h :\N_0 \times \N_0\to \N_0$  such that
for all $m,n \in \N_0$ if $W$ is a freely reduced word over the alphabet $\{a\}^{\pm 1} \cup \langle t \rangle$, representing the trivial element of $G$ and satisfying
$\sl(W) = m$ and $\|W\|_{\{a\} \cup \langle t \rangle}= n$, then
\[\area(W)\leq h(m,n).\]
(Since $\sl(W)  \leq \|W\|_{\{a\} \cup \langle t \rangle}$, the function $f:\N_0 \to \N_0$, $f(n)=\max\{h(m',n') \mid 0 \le m',n' \le n\}$ will serve as a
relative isoperimetric function of $G$ with respect to $\langle t \rangle$.)

The proof will use induction on $m$. By Britton's lemma, a freely reduced word $W$ of syllable length at most $1$ cannot represent the trivial element of $G$, hence we can define $h(0,n)=h(1,n)=0$, for all $n \in \N_0$.

Now suppose that $m \ge 2$ and the values of the desired function $h(s,n)$ have been found for all $s \in \{0,\dots,m-1\}$ and all $n \in \N_0$. Take any $n \in \N_0$.
If there are no freely reduced words $W$ such that  $\sl(W)=m$, $\|W\|_{\{a\} \cup \langle t \rangle}=n$ and $W  \eqG 1$ in $G$ then we set $h(m,n)=0$. Otherwise, let $W$ be such a word (in particular, $n \ge m \ge 2)$.

If $W$ is given by \eqref{eq:reduced_W} then, according to Lemma~\ref{lem:trim_chain_length}, $W$ admits a trimming chain
\[W=W_0 \red W_1\red \dots \red W_{\ell-1} \red W_\ell\]
at some place $i \in \{1,\dots,m-1\}$ of length $\ell \in \N$, where $\ell \le \max\{\nu_p(u_i),\nu_q(u_i)\}$.
Since $|u_i| \le n$, we see that $\ell \le \max\{\log_p(n),\log_q(n)\}=\log_r(n)$, where $r=\min\{p,q\}$.

Define $\beta=\log_r(\alpha)+1$ and observe that
\[\alpha^\ell \,n \leq \alpha^{\log_r(n)}\,n= r^{\log_r(\alpha) \log_r(n)}\,n  =n^{\log_r(\alpha)+1}=n^\beta,~\text { and}\]
\[\frac{\alpha^{\ell}-1}{\alpha-1} n \le \frac{1}{\alpha-1} \alpha^{\ell}\,n \le \frac{1}{\alpha-1} n^\beta.\]
Thus inequalities \eqref{eq:ineq_for_W_ell}, given by Lemma~\ref{lem:red_seq}, imply that
\begin{equation}\label{eq:W_ell_bounds_in_n}
\|W_\ell\|_{\{a\} \cup \langle t \rangle} \le n^\beta ~ \text{ and }~ \area(W) \le \area(W_\ell)+\frac{1}{\alpha-1} n^\beta .
\end{equation}

Since $m'=\sl(W_\ell)<\sl(W)$, by induction we have $\area(W_\ell) \le h(m',n')$, where $n'=\|W_\ell\|_{\{a\} \cup \langle t \rangle}$. In view of \eqref{eq:W_ell_bounds_in_n}, after defining
\[h(m,n)=\max\left\{ h(m',n')+\left\lfloor \frac{1}{\alpha-1} n^\beta \right\rfloor \,\middle |\, 0 \le m' \le m-1,~0 \le n' \le n^\beta \right \} \in \N_0,\]
we shall have $\area(W) \le h(m,n)$.

Thus we have found the required function $h:\N_0\times \N_0 \to \N_0$, so the proof is complete.
\end{proof}

\begin{remark} The argument from the proof of Theorem~\ref{thm:BS_suff} gives a double exponential upper bound for $\Delta_{G,\langle t \rangle}$:
\[\Delta_{G,\langle t \rangle}(n) \preceq n^{\beta^n}, ~\text{ for all } n \in \N_0,\]
where $\beta>1$ is the constant from that proof.
\end{remark}

\AtNextBibliography{\small}
\printbibliography

\end{document}

%% file: top.tex
\usepackage{amsmath} 
\usepackage{amssymb} 
\usepackage{amsthm} 
\usepackage[english]{babel} 
\usepackage[font=small,justification=centering]{caption} 
\usepackage[nodayofweek]{datetime}
\usepackage[shortlabels]{enumitem} 
\usepackage[T1]{fontenc} 
\usepackage[a4paper]{geometry} 
\usepackage[utf8]{inputenc} 
\usepackage{ifthen} 
\usepackage{mathabx} 
\usepackage{mathtools} 
\usepackage[dvipsnames]{xcolor} 
\usepackage[pdftex,  colorlinks=true]{hyperref} 
    \hypersetup{urlcolor=RoyalBlue, linkcolor=RoyalBlue,  citecolor=black}
\usepackage{setspace} 
\usepackage{tikz-cd} 
\usepackage{xfrac} 

\usepackage{amscd}
\usepackage{cleveref}
\usepackage[all]{xy}

\makeatletter
\def\@tocline#1#2#3#4#5#6#7{\relax
  \ifnum #1>\c@tocdepth 
  \else
    \par \addpenalty\@secpenalty\addvspace{#2}%
    \begingroup \hyphenpenalty\@M
    \@ifempty{#4}{%
      \@tempdima\csname r@tocindent\number#1\endcsname\relax
    }{%
      \@tempdima#4\relax
    }%
    \parindent\z@ \leftskip#3\relax \advance\leftskip\@tempdima\relax
    \rightskip\@pnumwidth plus4em \parfillskip-\@pnumwidth
    #5\leavevmode\hskip-\@tempdima
      \ifcase #1
       \or\or \hskip 1em \or \hskip 2em \else \hskip 3em \fi%
      #6\nobreak\relax
    \dotfill\hbox to\@pnumwidth{\@tocpagenum{#7}}\par
    \nobreak
    \endgroup
  \fi}
\makeatother

\newtheorem{thm}{Theorem}[section]
\newtheorem{theorem}[thm]{Theorem} \newtheorem{proposition}[thm]{Proposition} 
\newtheorem{lemma}[thm]{Lemma}
\newtheorem{corollary}[thm]{Corollary}
\newtheorem{prop}[thm]{Proposition}

\newtheorem{thmx}{Theorem}

\theoremstyle{definition}
\newtheorem{defn}[thm]{Definition}
\newtheorem{definition}[thm]{Definition}
\newtheorem{remark}[thm]{Remark}
\newtheorem{question}[thm]{Question}
\newtheorem{convention}[thm]{Convention}

\newtheorem{example}{Example}

\DeclareMathOperator{\Comm}{\mathrm{Comm}}

\DeclareMathOperator{\dist}{\mathsf{dist}}
\DeclareMathOperator{\Hdist}{\mathsf{hdist}}

\newcommand{\Len}{\rm{Len}}
\newcommand{\Area}{\rm{Area}}

\newcommand{\cala}{{\mathcal{A}}}
\newcommand{\calb}{{\mathcal{B}}}

\newcommand{\call}{{\mathcal{L}}}

\newcommand{\caln}{{\mathcal{N}}}

\newcommand{\calp}{{\mathcal{P}}}
\newcommand{\calP}{{\mathcal P}}
\newcommand{\calq}{{\mathcal{Q}}}
\newcommand{\calr}{{\mathcal{R}}}
\newcommand{\cals}{{\mathcal{S}}}
\newcommand{\calt}{{\mathcal{T}}}

\newcommand{\NN}{\mathbb{N}}

\usepackage{tikz}
\usetikzlibrary{arrows,quotes}
\tikzstyle{blackNode}=[fill=black, draw=black, shape=circle]
\usetikzlibrary{decorations.pathmorphing}
\tikzset{snake it/.style={decorate, decoration=snake}}
\usetikzlibrary{calc}

\newcommand{\nclose}[1]{\ensuremath{\langle\!\langle#1\rangle\!\rangle}}

\usepackage{centernot}
\newcommand{\N}{\mathbb{N}}
\newcommand{\Z}{\mathbb{Z}}
\renewcommand{\sl}{\mathrm{sl}}

\newcommand{\redone}{\mathop{\longrightarrow}\limits_{i}^{1}}
\newcommand{\redtwo}{\mathop{\longrightarrow}\limits_{i}^{2}}
\newcommand{\red}{\mathop{\longrightarrow}\limits_{i}}

\newcommand{\eqG}{\mathop{=}\limits^{G}}
\newcommand{\area}{\mathrm{Area}^{rel}}